%% file: MAIN.tex
\documentclass[10pt,oneside,american,reqno]{amsart}
\usepackage[a4paper]{geometry} 
\usepackage{setspace}
\usepackage{amssymb}
\usepackage[foot]{amsaddr}
\usepackage{mathtools}
\usepackage{todonotes}
\usepackage{amsthm, bm}
\usepackage[utf8]{inputenc}
\usepackage[autostyle=true]{csquotes}
\usepackage{color, colortbl}
\usepackage{lscape}
\usepackage{textgreek}
\usepackage{hyperref}
\hypersetup{colorlinks,
            citecolor=blue,
            urlcolor=blue,
            linkcolor=blue}
\usepackage{cleveref}
\usepackage[depth=subsection]{bookmark}
\usepackage{dsfont}
\usepackage{bbm}
\usepackage{subcaption}
\usepackage{nicefrac}
\usepackage{enumitem}
\usepackage{refcount}
\usetikzlibrary{decorations.pathmorphing}

\theoremstyle{plain}
\newtheorem{theorem}{Theorem}[section] 
\newtheorem{lemma}[theorem]{Lemma}
\newtheorem{corollary}[theorem]{Corollary}
\newtheorem{prop}[theorem]{Proposition}
\theoremstyle{definition}
\newtheorem{example}[theorem]{Example}

\newtheorem{conj}[theorem]{Conjecture}

\newtheorem{remark}[theorem]{Remark}

\theoremstyle{remark}

\makeatletter
\newenvironment{taggedsubequations}[1]
 {\begin{subequations}
  \def\@currentlabel{#1}
  \renewcommand{\theequation}{#1.\arabic{equation}}%
 }
 {\end{subequations}
 \ignorespacesafterend}
\makeatother

\input{macros}

\usepackage[backend=biber,
citestyle=numeric, 
bibstyle=numeric, 
hyperref=true, 
maxbibnames=20, 
maxcitenames=6, 
uniquelist=false,
url=false, 
doi=true, 
firstinits=true]{biblatex}
\DeclareNameAlias{author}{last-first}

\bibliography{literatur.bib}

\begin{document}
\title[Robust static and dynamic maximum flows]%
{Robust static and dynamic maximum flows}

\author[C. Biefel, M. Kuchlbauer, F. Liers, L. Waldmüller]%
{Christian Biefel, Martina Kuchlbauer, Frauke Liers, Lisa Waldmüller}
\date{\today}
\address[C. Biefel, M. Kuchlbauer, F. Liers, L. Waldmüller]{%
  Friedrich-Alexander-Universität Erlangen-Nürnberg,
  Germany;
  Cauerstr. 11,
  91058 Erlangen,  
  Germany}
\email{\{christian.biefel,martina.kuchlbauer,frauke.liers,lisa.waldmueller\}@fau.de}

\begin{abstract}
  \input{abstract}
\end{abstract}
\keywords{\input{keywords}}

\makeatletter
\@namedef{subjclassname@2020}{%
	\textup{2020} Mathematics Subject Classification}
	\def\blfootnote{\gdef\@thefnmark{}\@footnotetext}
\makeatother

\subjclass[2020]{\input{msc2010}}
\maketitle

\input{introduction}
\input{definitions}

\input{static}

\input{dynamic}

\input{conclusion}

\input{acknowledgements} 

\section*{References}
\printbibliography[heading=none]

\vspace{1cm}
\input{appendix}

\end{document}

%% file: macros.tex
\newcommand{\abs}[1]{\vert {#1} \vert}

\newcommand{\defsep}{\colon}

\newcommand{\st}{\text{s.t.}}

\newcommand{\R}{\mathbb{R}}

\newcommand{\N}{\mathbb{N}}
\newcommand{\Z}{\mathbb{Z}}
\newcommand{\Q}{\mathbb{Q}}
\newcommand{\T}{\mathcal{T}}
\newcommand{\Sc}{\mathcal{S}}

\newcommand{\fa}{\text{ for all }}

\newcommand{\Pa}{\mathcal{P}}
\newcommand{\barP}{\bar{\mathcal{P}}}

\newcommand{\tildex}{\Tilde{x}}

\newcommand{\stpath}{$s$-$t$-path}
\newcommand{\stpaths}{$s$-$t$-paths}

\newcommand{\vwpath}{$v$-$w$-path}
\newcommand{\vwpaths}{$v$-$w$-paths}

\newcommand{\doutP}{\delta^+_{\barP}}
\newcommand{\dinP}{\delta^-_{\barP}}
\newcommand{\doutA}{\delta^+_{A}}
\newcommand{\dinA}{\delta^-_{A}}
\newcommand{\dinPt}{\delta^-_{\barP}(t)}

\newcommand{\xoptnom}{x^*_{\hspace{-0.3mm}\textit{nom}}}
\newcommand{\foptnom}{f^*}

\newcommand{\fnom}{f}

\newcommand{\SGM}{\eqref{static:general}}
\newcommand{\SPM}{\eqref{static:path}}
\newcommand{\SAM}{\eqref{static:edge}}
\newcommand{\foptam}{f^*_{\hspace{-0.3mm}\textit{am}}}
\newcommand{\foptpm}{f^*_{\hspace{-0.3mm}\textit{pm}}}
\newcommand{\foptgm}{f^*_{\hspace{-0.3mm}\textit{gm}}}

\newcommand{\foptpmnom}{\fnom(\xoptpm)}
\newcommand{\foptgmnom}{\fnom(\xoptgm)}
\newcommand{\xoptam}{x^*_{\hspace{-0.3mm}\textit{am}}}
\newcommand{\xoptpm}{x^*_{\hspace{-0.3mm}\textit{pm}}}
\newcommand{\xoptgm}{x^*_{\hspace{-0.3mm}\textit{gm}}}
\newcommand{\xam}{x_{\hspace{-0.3mm}\textit{am}}}
\newcommand{\xpm}{x_{\hspace{-0.3mm}\textit{pm}}}
\newcommand{\xgm}{x_{\hspace{-0.3mm}\textit{gm}}}

\newcommand{\DGM}{\eqref{dynamic:general}}
\newcommand{\DPM}{\eqref{dynamic:path}}
\newcommand{\DAM}{\eqref{dynamic:arc}}
\newcommand{\foptdam}{f^*_{\hspace{-0.3mm}\textit{dam}}}
\newcommand{\foptdpm}{f^*_{\hspace{-0.3mm}\textit{dpm}}}
\newcommand{\foptdgm}{f^*_{\hspace{-0.3mm}\textit{dgm}}}
\newcommand{\foptdpmtr}{f^*_{\hspace{-0.3mm}\textit{tr}}}
\newcommand{\foptti}{f^*_{\hspace{-0.3mm}\textit{ti}}}
\newcommand{\foptdamnd}{\fnom(\xoptdam)} 
\newcommand{\foptdpmnd}{\fnom(\xoptdpm)}
\newcommand{\foptdgmnd}{\fnom(\xoptdgm)}
\newcommand{\xoptdam}{x^*_{\hspace{-0.3mm}\textit{dam}}}
\newcommand{\xoptdpm}{x^*_{\hspace{-0.3mm}\textit{dpm}}}
\newcommand{\xoptdgm}{x^*_{\hspace{-0.3mm}\textit{dgm}}}

\newcommand{\xoptti}{x^*_{\hspace{-0.3mm}\textit{ti}}}
\newcommand{\xdam}{x_{\hspace{-0.3mm}\textit{dam}}}
\newcommand{\xdpm}{x_{\hspace{-0.3mm}\textit{dpm}}}
\newcommand{\xdgm}{x_{\hspace{-0.3mm}\textit{dgm}}}

\newcommand{\thea}{\theta_{\textit{ea}}}

\renewcommand{\epsilon}{\varepsilon}
\newcommand{\uproman}[1]{\uppercase\expandafter{\romannumeral#1}}

\DeclareMathOperator*{\argmax}{arg\,max}

\newcommand{\hs}[1]{\hspace{#1}}

%% file: abstract.tex
We study the robust maximum flow problem and the robust maximum flow over
time problem where a given number of arcs $\Gamma$ may fail or may be delayed. 
Two prominent models have been introduced for these problems: either 
one assigns flow to arcs fulfilling weak flow conservation in
any scenario, or one assigns flow to paths where an arc failure or delay affects a 
whole path. 
We provide a unifying framework by presenting novel general models, 
in which we assign flow to subpaths. 
These models contain the known models as special cases and unify their
advantages in order to obtain less conservative robust solutions. 

We give a thorough analysis with respect to complexity of the general models.
In particular, we show that the general models are essentially NP-hard, 
whereas, e.g. in the static case with $\Gamma=1$ an optimal solution can be 
computed in polynomial time.
Further, we answer the open question about the complexity of the dynamic path model for $\Gamma=1$. 
We also compare the solution quality of the different models. 
In detail, we show that the general models have better
robust optimal values than the known models and we prove bounds on these gaps.

%% file: keywords.tex
robust maximum flows; network flows; flows over time; robust optimization%

%% file: msc2010.tex
05C21, 
90C05, 
90C17, 
90C27, 
90C35

%% file: introduction.tex
\section{Introduction}
In this work, we study the robust maximum flow problem in a network with arc 
failures or delays on arcs. 
An optimal solution to this problem maximizes the flow reaching a sink $t$ from a 
source $s$ in the case that at most $\Gamma$ arcs might fail or be delayed, following 
the prominent $\Gamma$-approach in robust optimization \cite{bertsimassim2003}.
The robust maximum flow problem under arc failure has been subject of 
various publications in which different models are discussed.
In \cite{aneja}, the authors present a polynomial time algorithm for
computing a solution maximizing the remaining flow when one arc fails. 
This model is extended to multiple arc failures 
by considering the \stpath ~decomposition of a flow in \cite{du2007maximum}. 
In this model, an arc failure has the effect that the flow on all paths that 
contain the failing arc is deleted. 
The goal is to find a flow and its \stpath ~decomposition such that the guaranteed remaining flow is maximized even if $\Gamma$ arcs fail. 
We call this formulation path model \SPM ~in the following.
For ${\Gamma=2}$, \cite{du2007maximum} claimed that the problem is
NP-hard. However, \cite{disser} pointed out a flaw in the arguments, 
so that the complexity of \SPM ~for fixed $\Gamma\geq 2$ is still open. 
Further, \cite{disser} show that the problem is NP-hard if $\Gamma$ 
is part of the input.
Another model for the robust maximum flow problem, which we call arc model \SAM, is given in \cite{bertsimas}:
flow is assigned to arcs and has to satisfy weak flow conservation 
at every node in every possible scenario, which makes this model quite restrictive. 
However, a great advantage of this model is that an optimal solution can be found in 
polynomial time for arbitrary $\Gamma$. 
Apart from that, \cite{bertsimas} introduce an adaptive model and also 
propose an approximation for \SPM, which was later improved in 
\cite{bertsimasRandomization}. 

Our main contribution to this setting is the introduction of a new general model 
\SGM, which combines the properties of the arc model \SAM ~and the path 
model \SPM ~and which covers them as special cases:
instead of assigning flow to arcs or \stpaths, we assign flow to subpaths. 
Therefore, \SGM ~extends our understanding of robust maximum flows and, in 
particular, on the connection between \SPM ~and \SAM, which seemed to 
be unrelated.
We extend most of the NP-hardness results for \SPM ~given in \cite{disser} 
to \SGM. 
We also show that \SGM ~remains solvable in polynomial time in 
many special cases for which \SPM ~is solvable in polynomial time, 
for example if only one arc may fail. 
Additionally, we provide comparisons 
of the robust optimal values of the three
models and present instances with specific graph structures for which \SGM ~indeed yields better robust optimal values, highlighting the relevance of the new model.

The flows in the former setting are independent of time and are thus called static flows. 
However, in many applications a time component has to be considered. 
In detail, every arc in the network is associated with a transit time, or travel time, 
the flow needs to traverse the arc.
We extend our study to these dynamic flows, in the literature often also called flows over time. 
In the nominal dynamic maximum flow problem, 
the goal is to maximize the flow that reaches the sink within a given time horizon. 
This problem was first introduced by Ford and Fulkerson in \cite{fordfulkersonDynamic}. 
For a more recent overview, we refer to e.g. \cite{skutella}.
To the best of our knowledge, there are only two publications 
on the theory of robust dynamic maximum flows.
In \cite{bertsimas}, an adaptive robust model is introduced for the case that arcs 
might fail.
In \cite{gottschalk}, one instead assumes that at most $\Gamma$ arcs are delayed. 
Clearly, if this delay is chosen to be greater than the time horizon, this corresponds to 
arc failure. 
The model in \cite{gottschalk}, which we call dynamic path model \DPM, 
is path-based and can be seen as a generalization of \SPM ~to flows over time. 
\cite{gottschalk} prove several results on the complexity of \DPM ~and they also show 
that temporally repeated flows, a solution concept in the 
nominal case, are not optimal in the robust setting but satisfy certain approximation guarantees.

In the dynamic setting, the path model \DPM ~is quite restrictive due to the strict
flow conservation and capacity constraints for any possible delay and therefore yields 
very conservative solutions as mentioned in \cite{gottschalk}.
We address this issue by introducing two new models, the dynamic arc model \DAM ~and 
the dynamic general model \DGM. 
These models are analogous extensions of their static counterparts.
Again, we show that \DGM ~contains \DPM ~and \DAM ~as special cases
and therefore is a relaxation of them.
In terms of the computational complexity of the models, we answer an open question 
from \cite{gottschalk} by proving that \DPM ~is already NP-hard for one delayed arc, 
which is in contrast to the polynomial time solvability in the static case.
Further, we clarify the complexity of the new models. 
While \DGM ~is NP-hard in all considered cases, \DAM ~is solvable in polynomial 
time regardless of the number of delayed arcs.
As in the static case, we compare the robust values of all three models
and show that \DGM ~yields the best robust optimal values on all instances
compared to the other models.
Thus, we extend the analysis of robust dynamic maximum flows to a 
similar basis as for the static case. 
\\

The robust maximum flow problem under arc failure can be seen as the opposite version of 
the network interdiction problem. 
While in the robust flow problem the maximum flow is sought under the knowledge 
that $\Gamma$ arcs will be deleted, in the network interdiction problem the task is to 
find a subset of $\Gamma$ arcs whose removal minimizes the maximum flow value in 
the remaining network, cf. \cite{wollmer,wood}, or \cite{chestnut} for recent 
hardness results.
Another related problem is the so called maximum 
multiroute flow problem, cf. \cite{kishimoto1996}. 
A solution to this problem is expressed as the conic combination of 
flows on arc-disjoint paths and 
thus the failure of an arc only destroys a certain fraction of the total flow. 
\cite{baffier14} obtain a $(\Gamma+1)$-approximation to \SPM ~with this concept.
Establishing appropriate robust flow models and clarifying 
algorithmic complexities is also relevant for the study of
more general tasks where a minimum-cost robust layout of a 
network is to be determined, cf.~\cite{mattia,cacchiani}.
\\

This paper is structured as follows:
in the next section we state necessary definitions and notation that 
we will use throughout the paper.
In \Cref{sec:static}, we consider robust static maximum flows. 
We present the two known robust maximum flow models, \SPM ~and \SAM, 
in \Cref{sec:static:models} and introduce the new model \SGM.
We prove several results on the complexity of \SGM ~in 
\Cref{sec:static_general_complexity}.
In \Cref{sec:static-bounds}, we compare the solution quality 
of the three models. 
The section on robust dynamic maximum flows, \Cref{sec:dynamic}, 
is structured similarly. 
In \Cref{sec:dynamic:models}, we discuss \DPM ~
and state the new models, \DAM ~and \DGM. 
We present complexity results for the three models in 
\Cref{sec:dynamic:complexity}. 
In \Cref{sec:dynamic:bounds}, we compare the solution quality 
of the models, 
followed by a brief discussion of possible approximation approaches for
\DGM ~in \Cref{sec:dyn:approx}.
Finally, in \Cref{sec:conclusion}, we collect open questions and discuss further
research directions.

%% file: definitions.tex
\section{Definitions and notation}\label{sec:Definitions}
We consider a directed graph $G=(V,A,s,t,u)$ with two distinct nodes in $V$, 
a source $s$ and a sink $t$, and positive integral capacities $u_a\in\N$ for all arcs $a \in A$.
In all figures, if not stated otherwise, arc labels denote the arc name and the capacity of the arc as depicted below. 
\begin{figure}[h]
		\begin{tikzpicture}
    	\node(v) at (0,0) [circle,draw=black, thick, minimum size=6mm] {$v$};
    	\node(w) at (3,0) [circle,draw=black, thick, minimum size=6mm] {$w$};
        \draw[thick,->](v) edge node[above]{$(a,u_a)$}(w);
        \end{tikzpicture}\label{fig:arclabelsStatic}
\end{figure}

The source $s$ has no incoming arcs whereas the sink $t$ has 
no outgoing arcs and additionally there is a directed $s$-$v$-$t$-path for every node $v\in V\setminus\{s,t\}$.

In the following, we assume that at most $\Gamma\in\N$ arcs may fail or 
be delayed.
We define the resulting set of all possible scenarios by
\begin{align*}
  \mathcal{S}\coloneqq\{S\subseteq A\defsep|S|\leq \Gamma\}
\end{align*}
and its corresponding indicator set for the considered value of $\Gamma$ by
\begin{equation*}
  \Lambda \coloneqq \{z \in \{0,1\}^{|A|} \defsep 
    \sum_{a\in A} z(a) \leq \Gamma\}.
\end{equation*}
We denote the set of all simple \stpaths ~by $\mathcal{P}$ and 
the set of all subpaths in $\mathcal{P}$ by 
\begin{align*}
  \barP\coloneqq\{\text{\vwpath~}P\defsep v,w\in V,~\exists~\text{\stpath~}P'~
    \st~P\subseteq P'\}.
\end{align*}

We use the standard notation to denote the set of incoming arcs at a node $v$ by 
$\dinA(v)$ and the set of outgoing arcs by $\doutA(v)$.
Analogously, we write $P\in\doutP(v)$ or $P\in\dinP(w)$ 
if the path $P\in\barP$ starts at $v$ or ends at $w$, respectively.

We denote flow by a vector $x$. 
The dimension of $x$ depends on the currently considered model and can be the 
number of arcs, the number of \stpaths ~or the number of subpaths in the graph. 

The nominal value of a flow is the amount that enters $t$ if no arc 
fails, i.e. for the example of $x\in \R^{|\barP|}_{\geq 0}$ this is given by
\begin{align*}
  \fnom(x)\coloneqq\sum_{P\in\dinPt}x(P).
\end{align*}
The nominal maximum flow problem maximizes this value satisfying strict 
flow conservation at the nodes and capacity constraints at the arcs:
	\begin{alignat*}{2}
	\foptnom~\coloneqq~\max_{\mathclap{x\in\R^{|A|}_{\geq 0}}} \quad
		&\sum_{a\in\dinA(t)}x(a) &&\\
	\st \quad & \sum_{a\in \dinA(v)} x(a)=\sum_{a\in \doutA(v)} x(a), \qquad
		&&\forall v\in V\setminus \{s,t \}, \\
		& \quad x(a)\leq u_a, &&\forall a\in A.
	\end{alignat*}

We hence denote the nominal maximum flow value on a graph by $\foptnom$.  
Note that relaxing the strict flow conservation to weak flow conservation,
meaning that the inflow in each node can be larger than the outflow, 
does not improve the optimal value. Further, for any flow defined on arcs, an 
\stpath ~decomposition can be found in polynomial time.
For an extensive overview on flow problems and algorithms, we refer to \cite{ahuja}.

Further, we write in the following for a positive integer 
$n\in\N$, $[n]\coloneqq\{1,...,n\}$, and set
${[n]\coloneqq\emptyset \text{ for all } n \in \Z\setminus\N}$.

%% file: static.tex
\section{Robust maximum flows}\label{sec:static}
In the following, we present different models for the robust maximum flow problem under arc failure, investigate their complexity and study solution quality gaps as well as the price of robustness.

\subsection{Robust flow models}\label{sec:static:models}

\subsection*{Path model (PM)}\bookmarksetup{depth=part}
One of the first maximum flow models dealing with arc failures was introduced in \cite{aneja}. 
There, the goal is to maximize the flow under the assumption that one arc may fail. 
The model was later generalized in \cite{du2007maximum} to the case of $\Gamma\geq 1$.
The overall flow is decomposed into flow on \stpaths, so that the flow variable $x$ has one entry per \stpath. 
If an arc fails, the flow on each path that contains the arc is deleted. 
The corresponding problem, which we call path model \SPM, reads as follows. 
\begin{taggedsubequations}{PM}\label{static:path}
  \begin{alignat}{2}
    \foptpm~:=~\max_{\mathclap{x\in \R^{|\Pa|}_{\geq 0}}} \quad 
    	&\sum_{P\in\Pa}x(P)-
      	\max_{S\in\Sc}\sum_{\substack{P\in\Pa:\\P\cap S\neq \emptyset}}x(P)
    	\qquad &&\\
    \st \quad &\sum_{\substack{P\in\Pa:\\a\in P}}x(P)\leq u_a, &&\forall a\in A.
  \end{alignat}
\end{taggedsubequations}
While for $\Gamma=1$, this problem is solvable in polynomial time, 
cf. \cite{aneja}, it is NP-hard to find an optimal solution if $\Gamma$ is part
of the input, cf. \cite{disser}. For fixed $\Gamma\geq 2$, the complexity of the model is unknown.

\subsection*{Arc model (AM)}\label{subsec:SAM} in 
In the arc model introduced in \cite{bertsimas}, flow is assigned to each arc 
so that the flow variable $x$ has $\abs{A}$ entries. 
Weak flow conservation has to hold at every node and for every
possible scenario of at most $\Gamma$ many arc failures. 
It is hence required that every node has non-negative outflow even if
$\Gamma$ of the incoming arcs fail. 
We call this property robust flow conservation in the following.
The corresponding arc model \SAM ~is given by

\begin{taggedsubequations}{AM}\label{static:edge}
	\begin{alignat}{2}
	  \foptam~:=~\max_{\mathclap{x\in\R^{|A|}_{\geq 0}}} \quad 
	  		&\sum_{a\in\delta^-_A(t)}x(a)
	    	-\max_{z\in\Lambda} \sum_{a\in\delta^-_A(t)}z(a)x(a) &&\\
	  \st \quad & \sum_{a\in \delta_{A}^{-}(v)} (1-z(a)) x(a) 
	      	-\sum_{a\in\delta_{A}^{+}(v)} x(a) \geq 0, \qquad
	      	&&\forall v\in V\setminus \{s,t\},z \in \Lambda,
	      	\label{static:edge-rob-constraint}\\
	   & \quad x(a) \leq u_a, &&\forall a\in A.
	\end{alignat}
\end{taggedsubequations}
\cite{bertsimas} show that \eqref{static:edge} is equivalent to a linear program, where the number of variables and constraints is polynomial in 
$|V|$ and $|A|$. Therefore, an optimal solution to \SAM ~can be computed in polynomial time.

\subsection*{General model (GM)}
We now introduce a new general model that combines the properties of \SAM ~and \SPM.
The flow is assigned to paths in the set $\barP$ of all possible subpaths of \stpaths. 
As in the path model, we assume that the flow on all paths that contain a failing arc 
is deleted. 
Additionally, we require weak flow conservation and hence non-negative outflow at
every node in every possible scenario as in the arc model, i.e. robust flow 
conservation as in \SAM. 
The general model \SGM ~reads as follows.

\begin{taggedsubequations}{GM}\label{static:general}
  \begin{alignat}{2}
    \foptgm~:=~\max_{\mathclap{x\in\R^{|\barP|}_{\geq 0}}} \qquad
    	&\sum_{\mathclap{P\in\dinPt}} ~~~x(P)-\max_{S\in\mathcal{S}}
      		\sum_{\substack{P\in\dinPt:\\P\cap S\neq\emptyset}}x(P) &&\\
    \st \qquad 
    	& \sum_{\mathclap{P\in \delta_{\barP}^{-}(v)}} ~~~x(P) 
	      -\sum_{\substack{P\in\delta_{\barP}^{-}(v):\\P\cap S\neq\emptyset}}x(P) 
	      \geq \sum_{P\in\delta_{\barP}^{+}(v)} x(P), ~\quad
	    &&\forall v\in V\setminus \{s,t \}, S\in\mathcal{S},
	      \label{static:general-rob-constraint}\\
    	& \sum_{\mathclap{\substack{P\in\barP:\\a\in P}}}~~~x(P)\leq u_a, 
    	&&\forall a\in A.
    	  \label{static:general-capacity}
  \end{alignat}
\end{taggedsubequations}
\addtocounter{equation}{-3}

Given a feasible solution $x$ to \eqref{static:general}, 
we denote its objective value 
by $f_R(x)$. 
We denote an optimal solution by $\xoptgm$ with optimal value $\foptgm=f_R(\xoptgm)$. 
We observe that any feasible solution $\xpm$ to \SPM ~can be easily 
extended to a feasible solution to \SGM ~by setting 
$$\xgm(P):=\begin{cases} \xpm(P), &P\in\Pa,\\0, &\text{else.} \end{cases}$$
For ease of notation, we will 
assume that $\xpm\in\R^{|\Pa|}$ 
when we refer to \eqref{static:path} and $\xpm\in\R^{|\barP|}$ when we refer to 
\eqref{static:general}. It follows that $\foptpm=f_R(\xoptpm)$, 
where $\xoptpm$ denotes an optimal solution to \SPM. 
Similarly, all feasible solutions to \SAM ~can be extended to feasible 
solutions to \SGM. 
Again, we have $\foptam=f_R(\xoptam)$, where $\xoptam$ denotes an optimal solution 
to \SAM.
It follows that the new model \SGM ~in fact is a generalization of 
the two previous models since they are special cases in which the solutions 
are restricted to \stpaths ~or to paths consisting of single arcs, respectively. 

By adding an additional variable $\lambda$, the max-min-problem \eqref{static:general} can
be written as the linear program
\begingroup
  \allowdisplaybreaks	
\begin{subequations}\label{static:general-lambda}
  \begin{alignat}{3}
    \foptgm~=~\max_{\mathclap{x\in\R^{|\barP|}_{\geq 0},\lambda\in\R}} 
      ~~\qquad &\sum_{\mathclap{P\in\dinPt }}\quad && x(P)-\lambda &&
      \label{static:general-lambda-obj}\\
    \st~~\qquad & \sum_{\mathclap{\substack{P\in\dinPt :\\P\cap S\neq\emptyset}}}
    	&& x(P)-\lambda \leq 0, \qquad
    	&&\forall S\in\mathcal{S},
    	\label{static:general-lambda-lambdaconstr}\\
    \quad & \sum_{\mathclap{P\in \delta_{\barP}^{-}(v)}} && x(P) 
	      -\sum_{\substack{P\in\delta_{\barP}^{-}(v):\\P\cap S\neq\emptyset}}x(P) 
	      \geq \sum_{P\in\delta_{\barP}^{+}(v)} x(P), ~\quad
	    &&\forall v\in V\setminus \{s,t \}, S\in\mathcal{S},
	    \label{static:general-lambda-robconstr}\\
    \quad & \sum_{\substack{P\in\barP:\\a\in P}}  && x(P)\leq u_a, 
    	&&\forall a\in A.
    	\label{static:general-lambda-cap}
  \end{alignat}
\end{subequations}
\endgroup
For an optimal solution $(x^*,\lambda^*)$ to \eqref{static:general-lambda}, 
it holds that $x^*$ is an optimal solution to \SGM ~and 
${\lambda^*=\max_{S\in\mathcal{S}}\sum_{P\in \dinPt :P\cap S \neq\emptyset}x^*(P)}$ is the amount of flow that is deleted in the worst case.
Clearly, the same rewriting can also be applied to \SPM ~and \SAM.


\bookmarksetup{depth=subsection}
\subsection{Complexity of \SGM}
\label{sec:static_general_complexity}
We now analyze the complexity of \SGM ~for different choices of $u$ and $\Gamma$. 
We start with the case $\Gamma = 1$, in which only one arc may fail, and prove polynomial time solvability.
Further, we extend several hardness results from \cite{disser} to our setting. Thereby, we exploit the fact that \SPM ~is a special 
case of the general model.  
An overview of the complexity of the three robust flow models can be found 
in \Cref{tab: Complexity overview static}.
\begin{table}[h]
\centering
\resizebox{1\textwidth}{!}{
\begin{tabular}{l|l|l|l}
                & Arc Model \SAM & Path Model \SPM & General Model \SGM \\
    \hline & & &\\
    $\Gamma=1$ 
    	& poly time \cite{bertsimas} 
    	& poly time \cite{aneja} 
    	& poly time (Thm. \ref{thm:static:general-gamma1})\\
    & & & \\
    fixed $\Gamma\geq 2$ 
    	& poly time \cite{bertsimas} & ? & ? \\
    & & & \\
    $\Gamma$ arb. 
    	& poly time \cite{bertsimas} 
    	& strongly NP-hard \cite{disser} 
    	& strongly NP-hard (Thm. \ref{thm:static:general-gammainput}) \\
    & & & \\
    $\Gamma$ arb., $u\equiv 1$ 
    	& poly time \cite{bertsimas} 
    	& poly time \cite{disser} 
    	& poly time (Prop. \ref{lem:SGM_u=1_poly})\\
    & & & \\
    $\Gamma$ arb., $u\in \{1,\infty\}$ 
    	& poly time \cite{bertsimas} 
    	& NP-hard \cite{disser} 
    	& NP-hard (Prop. \ref{lem:SGM_capacity1infty_NP})
\end{tabular}}\vspace{0.3cm}
    \caption{Overview of the complexity results of the 
    			robust static flow models.}
    \label{tab: Complexity overview static}
\end{table}

We first state an auxiliary result, which we use in the subsequent complexity proofs.
The result formalizes the following observation: 
if the number of incoming arcs of a node $w\in V$ is not larger than $\Gamma$, 
then there is a scenario in which all incoming arcs of~$w$ fail. 
Hence, in a feasible solution to \SGM, no flow can be sent on paths starting 
at~$w$ due to robust flow conservation \eqref{static:general-rob-constraint}. 
As a consequence, all flow sent on paths ending at $w$ can be discarded without 
reducing the robust flow value. 
\begin{lemma}\label{cor:static:aux-lemma}
	Let $|\delta_{A}^-(w)| \leq \Gamma$ for a node $w\in V\setminus\{s,t\}$ 
	and let $x$ be a feasible solution to \SGM. 
	Then, $x(P)=0$ for all $P\in\doutP(w)$.
	Furthermore, if $x(P)>0$ for some $P\in\dinP(w)$, then the solution $x'$ defined by
	\begin{align*}
	x'(P):=\begin{cases}
	0, &P\in\delta_{\barP}^-(w),\\
	x(P), &P\in\barP\setminus\delta_{\barP}^-(w),
	\end{cases}
	\end{align*}
	is also feasible to \SGM ~and has the same robust flow value as $x$.
\end{lemma}

As mentioned above, the in general NP-hard path model is solvable in polynomial time 
for $\Gamma=1$. 
Although the general model at first glance seems more complex considering the set of variables and constraints, 
we now show that \SGM ~lies in the same complexity class and hence is solvable in polynomial time for $\Gamma=1$.

\begin{theorem}\label{thm:static:general-gamma1}
 For $\Gamma=1$, the optimal value of \SGM ~can be computed by solving the 
 following linear program of polynomial size
  \begingroup
  \allowdisplaybreaks	
  \begin{subequations}\label{eq:static-gamma1lp}
    \begin{align}
      \foptgm ~=\quad \max_{\substack{y\in\R^{|V|^2|A|}_{\geq 0},\\\nu\in \R}} \quad
      	& \sum_{a\in A}\sum_{v\in V} y(a,v,t)-\nu
      	\label{eq:static-gamma1lp-obj}\\[3pt]
      \st \quad &\sum_{v\in V}y(a,v,t) - \nu \leq 0, 
      	&& \forall a\in A,
        \label{eq:static-gamma1lp-nu}\\
      &\sum_{\mathclap{a\in\dinA(v')}}\quad\sum_{v\in V}y(a,v,v')
      	-\sum_{v\in V}y(a',v,v')	
      	&&\nonumber\\
      &\qquad\quad \geq \sum_{a\in\doutA(v')}\sum_{w\in V}y(a,v',w), 
        &&\forall v'\in V, a'\in A,
        \label{eq:static-gamma1lp-robconst}\\
      &\sum_{\mathclap{a\in\dinA(v')}}\quad y(a,v,w)
      	=\sum_{a\in\doutA(v')}y(a,v,w), \quad
        &&\forall v'\in V\setminus\{s,t\}, v,w \in V\setminus\{v'\},
        \label{eq:static-gamma1lp-strict}\\
      &\sum_{v\in V}\sum_{w\in V}y(a,v,w)\leq u_a, 
      	&& \forall a\in A.
        \label{eq:static-gamma1lp-cap}
    \end{align}
  \end{subequations}
  \endgroup
  
  Given a feasible solution $(y,\nu)$ to \eqref{eq:static-gamma1lp}, a
  feasible solution to \SGM ~with same objective value can be computed in polynomial
  time. 
  
  Hence, \SGM ~can be solved in time polynomial in the size of the 
  graph $G$ if $\Gamma=1$.
\end{theorem}
\begin{proof}
  First, let $(x,\lambda)$ be a feasible solution to problem  
  \eqref{static:general-lambda}, which is the linear programming
  equivalent of \SGM.
  We now construct a feasible solution $(y,\nu)$ to \eqref{eq:static-gamma1lp} 
  with same objective value.
  We set $\nu\coloneqq \lambda$ and 
  \begin{align*}
    y(a,v,w)\coloneqq \sum_{\substack{P\in(\doutP(v)\cap\dinP(w)):\\a\in P}}x(P) 
      \quad \forall a\in A, v,w \in V,
  \end{align*}
  i.e. $y(a,v,w)$ denotes the total flow on $a$ of paths starting at $v$ and ending
  at $w$.
  Thus, for any $a\in A$ we have 
  \begin{equation}\label{eq:gamma1lp_proof_1}
  \sum_{v\in V}y(a,v,t)= \sum_{\substack{P\in\dinP(t):\\a\in P}}x(P),
  \end{equation}
  and hence, $(x,\lambda)$ for \eqref{static:general-lambda} and $(y,\nu)$ 
  for \eqref{eq:static-gamma1lp} have the same objective value.
  Furthermore, 
  \eqref{static:general-lambda-lambdaconstr} implies \eqref{eq:static-gamma1lp-nu}.
  Additionally, in \eqref{eq:gamma1lp_proof_1} we can replace $t$ by any node
  $v'\in V$ and thus robust flow conservation \eqref{static:general-lambda-robconstr}
  for $x$ implies \eqref{eq:static-gamma1lp-robconst} for $y$. 
  Strict flow conservation \eqref{eq:static-gamma1lp-strict} for $y(\cdot,v,w)$
  holds due to the fact that $x$ also fulfills strict flow conservation within
  each path from $v$ to $w$. 
  Finally, $y$ is non-negative since $x$ is non-negative and $y$ fulfills the capacity 
  constraint \eqref{eq:static-gamma1lp-cap} due to \eqref{static:general-lambda-cap}. 
  Thus, $(y,\nu)$ is feasible for
  \eqref{eq:static-gamma1lp} with the same objective value as $(x,\lambda)$ for 
  \eqref{static:general-lambda}.
 
  On the other hand, let $(y,\nu)$ now be an arbitrary feasible solution 
  to \eqref{eq:static-gamma1lp}.
  Constraint \eqref{eq:static-gamma1lp-strict} ensures that $y(\cdot,v,w)$ is
  a flow from $v$ to $w$ satisfying strict flow conservation. 
  Therefore, there exists a decomposition of this $v$-$w$-flow into flow on 
  cycles and simple $v$-$w$-paths. 
  Let $x\in\R^{|\barP|}$ be the union of these path decompositions, for all pairs 
  $v,w\in V$, where we discard flow on cycles. Further, we set $\lambda\coloneqq\nu$. 
  We now show that $(x,\lambda)$ is feasible for \eqref{static:general-lambda}
  with the same objective value as $(y,\nu)$ in \eqref{eq:static-gamma1lp}. 
  From the definition of $x$ as path decomposition 
  it follows that $x$ is non-negative and fulfills the 
  capacity constraint \eqref{static:general-lambda-cap}.
  Furthermore, for any $a\in A$ we have 
  $\sum_{P\in\dinP(t):a\in P}x(P)=\sum_{v\in V}y(a,v,t)$ and therefore
  $\sum_{P\in\dinP(t)}x(P)=\sum_{a\in\dinA(t)}\sum_{v\in V}y(a,v,t)$. 
  Thus, $(x,\lambda)$ has the same objective value as $(y,\nu)$. 
  \eqref{static:general-lambda-lambdaconstr} holds due to 
  \eqref{eq:static-gamma1lp-nu}.
  Finally, due to strict flow conservation of $y(\cdot,v,w)$ for every pair 
  $v,w\in V$, for any node $v'\in V\setminus\{s,t\}$ it holds that 
  \begin{align*}
    \sum_{P\in\dinP(v')}x(P)-\sum_{P\in\doutP(v')}x(P)=
    \sum_{a\in\dinA(v')}\sum_{v\in V}y(a,v,v')
    -\sum_{a\in\doutA(v')}\sum_{w\in V}y(a,v',w).
  \end{align*}
  Additionally, by definition of $x$, for any $a\in A$ we have 
  $\sum_{v\in V}y(a,v,v')=\sum_{P\in\dinP(v'):a\in P}x(P)$ so that 
  \eqref{eq:static-gamma1lp-robconst} implies 
  \eqref{static:general-lambda-robconstr}. 
  Hence, $(x,\lambda)$ is feasible for \eqref{static:general-lambda}. 
  
  Summarizing, for every feasible solution $(y,\nu)$ to \eqref{eq:static-gamma1lp}
  there exists a feasible solution $(x,\lambda)$ to \eqref{static:general-lambda}
  with same objective value and vice versa. 
  Thus, \eqref{eq:static-gamma1lp} has an optimal value of $\foptgm$.
  Since a path decomposition of a flow can be computed in polynomial time, 
  we can furthermore compute an optimal solution $\xoptgm$ to \SGM ~in polynomial
  time given a feasible solution $(y,\nu)$ to \eqref{eq:static-gamma1lp}.
\end{proof}

In addition to the complexity of \SGM ~for $\Gamma=1$, we next show that in this 
case there is a nominal optimal solution which is robust optimal. 
A similar fact has been shown for adaptive robust flows in \cite{bertsimas}.

\begin{theorem}\label{lemma:static:gamma1-maxflow}
  For $\Gamma=1$, there is an optimal solution $\xoptgm$ to \SGM ~which is also 
  nominal optimal, i.e. $\fnom(\xoptgm)=\foptnom$.
\end{theorem}
The proof of \Cref{lemma:static:gamma1-maxflow} is quite technical and can be found in \Cref{prooflemma:static:gamma1-maxflow}.

The following result states NP-hardness of \SGM ~when 
$\Gamma$ is arbitrary.
\begin{theorem}\label{thm:static:general-gammainput}
  \SGM ~is strongly NP-hard if $\Gamma$ is part of the input.
\end{theorem}
\begin{proof}
	In \cite[Theorem 2, proof]{disser}, it is proven by a reduction from \textsc{clique}
	that \SPM ~is strongly NP-hard if $\Gamma$ is part of the input.
	We use the same construction and extend their idea to our setting:
	they construct an instance $I=(G=(V,A,s,t,u),\Gamma)$ 
	for \SPM ~from a given \textsc{clique} instance $I'=(G'=(V', A'),\Gamma')$ 
	and show that an optimal solution to \eqref{static:path} sends flow 
	on a specific arc $(v',v'')\in A$ if and only if 
	there exists a clique of size $\Gamma'$ in $G'$.
	
	Let $\xoptgm$ be an arbitrary optimal solution to \SGM ~on the instance $I$.
	In the following, we will observe that every node $v\in V\setminus\{t\}$ has less 
	than or equal to $\Gamma$ incoming arcs. 
	We will then apply \Cref{cor:static:aux-lemma} to construct an optimal solution 
	$\xoptpm$ to \SPM ~on $I$.
	First, we observe that every node $v\in V\setminus(B\cup\{s,t\})$, 
	where $B\subset V$ is defined as in \cite[Theorem 2, proof]{disser}, 
	has one or two incoming arcs.
	Second, from the construction of $G$ and the definition of $\Gamma$ as in 
	\cite[Theorem 2, proof]{disser}, we derive for the nodes $b\in B$: 
	\begin{align*}
		|\delta_{A}^-(b)| \leq 2+2|A'| \leq \Gamma.
	\end{align*}
		
	Thus, from \Cref{cor:static:aux-lemma} it follows that 
	$\xoptgm(P)=0$ for all \vwpaths ~$P$ in $G$ with $v\neq s$. 

	Additionally, it follows that we can construct a solution $x^*$ 
	\begin{align*}
	  x^*(P)\coloneqq\begin{cases}
	    \xoptgm(P), &P\in\Pa,\\
	    0, &P\in\barP\setminus\Pa,
	  \end{cases}
	\end{align*}
	with the same robust objective value as $\xoptgm$. The solution $x^*$ only 
	uses \stpaths ~and is therefore an optimal solution to \SPM.
	
	Summarizing, by solving \SGM ~on $I$, we can find in a time, which is 
	polynomial in the size of the original solution, an optimal solution to \SPM 
	~by which we can decide whether there is a clique of size $\Gamma'$ in $G'$. 
\end{proof}

\subsection*{Restrictions on the capacity \textit{u}}\bookmarksetup{depth=part}
\label{sec:static_general_special_cases}

In the following, we investigate the complexity of \SGM ~for several special cases  in which the capacity $u$ is restricted to certain values.
Again, the results are based on analogous statements in \cite{disser}.
We start with the special case in which all arcs have unit capacity.

\begin{prop}\label{lem:SGM_u=1_poly}
	Let $u \equiv 1$. 
	Then, for arbitrary $\Gamma\geq 1$, an optimal solution to \SGM ~can be determined 
	in polynomial time.
\end{prop}
\begin{proof}
	The proof works analogously to the proof of Theorem 7 in \cite{disser}.
	Let $C$ be the value of the minimal cut of a given instance. 
	At most $C$ units of flow can arrive at the sink in the nominal case. 
	Due to the unit capacities, at most $\Gamma$ units of flow may be deleted. 
	If $\Gamma > C$, then no flow arrives at the sink 
	in the robust case. If $C > \Gamma$, exactly $\Gamma$ of the $C$ 
	flow units are deleted and the robust flow value equals $C-\Gamma$. 
	Hence, a nominal optimal solution is also optimal for \eqref{static:general}.
\end{proof}

\begin{prop}\label{lem:SGM_capacity1infty_NP}
Let $\Gamma$ be part of the input and
let	capacities $u_a\in\{1, u_{\max}\}$ or $u_a\in\{1, \infty\}$ for all $a\in A$.	
Then, \SGM ~is NP-hard.
\end{prop}
The proof uses the analogous result for the problem with arbitrary capacities, following the proof for \SPM ~in \cite{disser}. 
For this, we exploit an idea that is similar to \Cref{cor:static:aux-lemma}, 
and which we apply in a similar way as in the proof of \Cref{thm:static:general-gammainput}.
We refer to \Cref{prooflem:SGM_capacity1infty_NP}.

\subsection*{Integral flows}
\label{sec:static-complexity-integral}
So far, all flows were allowed to take fractional values. 
In the nominal maximum flow problem, there always exists an optimal flow 
with only integral values, as long as the capacities are integral, 
cf. \cite[Theorem 6.5]{ahuja}. 
Such an optimal integral flow usually can be found by standard flow algorithms, 
e.g. the well known Ford-Fulkerson algorithm, 
cf. \cite{fordfulkersonStatic}, augments the flow by an integer value in each 
iteration.
In the various robust models, however, there is not necessarily an optimal solution 
with integral values. 
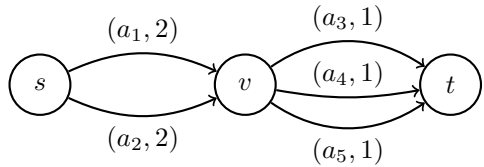
\begin{figure}[h]
	\begin{tikzpicture}[scale=.9]
      \node(s) at (-3,0) [circle,draw=black, thick, minimum size=8mm] {$s$};
      \node(v) at (0,0) [circle,draw=black, thick, minimum size=8mm] {$v$};
      \node(t) at (3,0) [circle,draw=black, thick, minimum size=8mm] {$t$};
      \draw[thick,->,bend left =25](s)edge node[above]{$(a_1,2)$}(v);
      \draw[thick,->,bend right=25](s)edge node[below]{$(a_2,2)$}(v);
      \draw[thick,->,bend left=35](v)edge node[above]{$(a_3,1)$}(t);
      \draw[thick,->,bend left=-10](v)edge node[above]{$(a_4,1)$}(t);
      \draw[thick,->,bend left=-35](v)edge node[below]{$(a_5,1)$}(t);
	\end{tikzpicture}
	\caption{Instance with $\foptam=\nicefrac{4}{3}$, 
			$\foptpm=\nicefrac{3}{2}$ and $\foptgm=2$ for $\Gamma=1$.}
	\label{fig:integralflow_uleq2}
\end{figure}
For example, on the instance in \Cref{fig:integralflow_uleq2}, the unique
optimal solution to \SAM ~sends $\nicefrac{2}{3}$ units of flow on each arc
from $v$ to $t$, every optimal solution to \SPM ~uses at least two \stpaths ~with 
fractional values, and the optimal solution to \SGM ~sends $\nicefrac{1}{3}$ 
units of flow on each \stpath ~as well as on each arc from $v$ to $t$. 

Therefore, it is of interest to study the complexity of the integral models
in which we require all entries of the flows to be integral.
To the best of our knowledge, there are no results on the complexity of 
integral \SAM. 
For integral \SPM, \cite{aneja} show that computing an optimal solution is 
possible in polynomial time if $\Gamma=1$. Moreover, \cite{disser} prove that 
an optimal integral flow can be found in polynomial time for arbitrary $\Gamma$ if 
the capacities are restricted to $u_a\leq 2$ for all $a\in A$.
The complexity of integral \SGM ~remains an open question for these special cases, 
whereby we note that the proof techniques by \cite{disser} are not directly 
transferable to integral \SGM.
On the other hand, \cite{disser} also show that integral \SPM ~is NP-hard and 
that there is no $(\nicefrac{3}{2}-\epsilon)$-approximation for integral \SPM 
~even for $\Gamma=2$ and instances with $u_a\leq 3$ for all $a\in A$. 
We extend this result to integral \SGM.   
    
\begin{theorem}\label{lem:SGM_integralSolution_NP}
	 Unless P = NP, there is no $(\nicefrac{3}{2}-\epsilon)$-approximation 
	 algorithm for computing an optimal solution to integral \SGM, even 
	 when restricted to instances where $\Gamma=2$ and $u_a\leq 3$ for all 
	 $a \in A$.
\end{theorem}
The proof extends the corresponding proof
for integral \SPM ~from \cite{disser} and can be found in \Cref{prooflem:SGM_integralSolution_NP}.
The proof utilizes the fact that, for an integral solution, single flow units 
cannot be distributed among multiple paths, which leads to a reduction from \textsc{arc-disjoint paths}.


\bookmarksetup{depth=subsection}
\subsection{Comparison of the solution quality between the robust flow models}
\label{sec:static-bounds}
As discussed in \Cref{sec:static:models}, all feasible solutions to 
\SAM ~and \SPM ~are feasible for \SGM. 
Therefore, it obviously holds that the optimal value of the general model is always greater than or equal to the optimal value of the arc model and the path model. 
We first show
that there is no bound on how much better the optimal value of 
\SPM ~and \SGM ~can be compared to \SAM, i.e. the extent to which the arc model is surpassed by the others is unbounded.
Second, we provide a lower bound on the factor by which the arc model as well as the general model may outperform the path model. 
Thus, there are instances where the path model achieves a greater optimal value, but also those where the arc model leads to a greater one, so that neither of the two models can be considered to be better in general.

Furthermore, we conjecture that the above mentioned lower bound on the gap between 
\SPM ~and \SGM ~is tight.
In \Cref{Prop:gapSGMSPMupperboundGamma1}, we prove this for the special case that $G$ is a DAG and only one arc may fail, i.e. $\Gamma=1$.

\begin{prop}\label{Prop:gapSPM_SAM}
  For any $\Gamma\in \N$ and $\alpha\in\R$ there are instances such that 
  $\foptpm>\alpha \foptam$ and consequently $\foptgm>\alpha \foptam$.
\end{prop}
\begin{proof}

  For an arbitrary $\Gamma\geq 1$, we construct an instance $I=(G,\Gamma)$ such 
  that $\foptpm=1$ and $\foptam=0$, cf. \Cref{fig:gapSPM_SAM}.
  The constructed graph $G$ consists of the nodes $s$, $t$, and $v_i$, 
  $i\in[\Gamma+1]$. 
  For all $i\in [\Gamma+1]$, $s$ is connected to $v_i$ by an arc $a'_i$ and 
  $v_i$ is connected to $t$ by an arc $a''_i$.  
  All arcs have unit capacity.

  In all feasible solutions $\xam$ to \SAM ~it holds that $\xam(a''_i)=0$ 
  for all $i\in[\Gamma+1]$ due to robust flow conservation 
  \eqref{static:edge-rob-constraint}.
  It follows that $\foptam=0$. 
  
  The set of \stpaths ~in $G$ is given by 
  $\Pa=\{P_i\coloneqq \{a'_i,a''_i\} ~ \forall i \in [\Gamma+1]\}$ 
  and the optimal solution $\xoptpm$ to \SPM ~is given 
  by $\xoptpm(P)=1 ~\forall P \in \Pa$. 
  Deleting $\Gamma$ arcs results in $\foptpm=(\Gamma+1)-\Gamma=1$, 
  proving the first claim. 
  Since the optimal solution to \SPM ~is also optimal for \SGM ~on this graph, the second claim follows.
\end{proof}

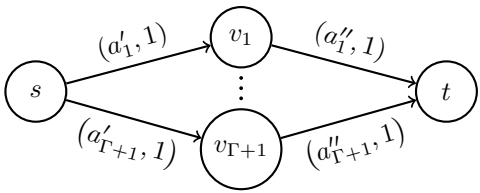
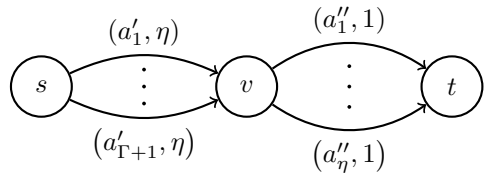
\begin{figure}[h]
  \begin{minipage}[t]{.49\textwidth}
    \begin{tikzpicture}[scale=.9]
      \node at (0,1.8) {}; 
      \node at (0,-1.8) {}; 
      \node(s) at (-3,0) [circle,draw=black, thick, minimum size=8mm]{$s$};
      \node(v1) at (0,0.8) [circle,draw=black, thick, minimum size=8mm]{$v_1$};
      \node at (0,0.2)[fill, circle, inner sep = 0.5pt] {};
      \node at (0,0.05)[fill,circle,inner sep=0.5pt]{};
      \node at (0,-0.1)[fill, circle, inner sep = 0.5pt] {};
      \node(vgamma1)at(0,-0.85)[circle,draw=black,thick,minimum size=8mm]
      		{$v_{\Gamma+1}$};
      \node(t) at (3,0) [circle,draw=black, thick, minimum size=8mm] {$t$};
      \draw[thick,->](s)edge node[above,rotate=15]{$\left(a'_1,1\right)$}(v1);
      \draw[thick,->](s)edge node[below,,rotate=-15]
      		{$\left(a'_{\Gamma+1},1\right)$}(vgamma1);
      \draw[thick,->](v1)edge node[above,rotate=-16]{$\left(a''_1,1\right)$}(t);
      \draw[thick,->](vgamma1)edge node[below,rotate=16]
      		{$\left(a''_{\Gamma+1},1\right)$}(t);
    \end{tikzpicture}
    \subcaption{Instance with $\foptpm=\foptgm=1$ and $\foptam=0$.}
    		\label{fig:gapSPM_SAM}
  \end{minipage}
  \begin{minipage}[t]{.49\textwidth}
    \begin{tikzpicture}[scale=.9]
      \node at (0,1.8) {}; 
      \node at (0,-1.8) {}; 
      \node(s) at (-3,0) [circle,draw=black, thick, minimum size=8mm] {$s$};
      \node(v) at (0,0) [circle,draw=black, thick, minimum size=8mm] {$v$};
      \node(t) at (3,0) [circle,draw=black, thick, minimum size=8mm] {$t$};
      \draw[thick,->,bend left =25](s)edge node[above]
      	{$\left(a'_{1},\eta\right)$}(v);
      \draw[thick,->,bend right=25](s)edge node[below]
      	{$\left(a'_{\Gamma+1},\eta\right)$}(v);
      \draw[thick,->,bend left =35](v)edge node[above]
      	{$\left(a''_{1},1\right)$}(t);
      \draw[thick,->,bend right=35](v) edge node[below]
      	{$\left(a''_{\eta},1\right)$} (t);
      \node at (-1.5,0.25)[fill, circle, inner sep = 0.5pt] {};
      \node at (-1.5,0)[fill,circle,inner sep=0.5pt]{};
      \node at (-1.5,-0.25)[fill, circle, inner sep = 0.5pt] {};
      \node at (1.5,0.30)[fill, circle, inner sep = 0.5pt] {};
      \node at (1.5,00)[fill,circle,inner sep=0.5pt]{};
      \node at (1.5,-0.30)[fill, circle, inner sep = 0.5pt] {};
    \end{tikzpicture}
    \subcaption{Instance with $\foptam=\foptgm=\eta-\Gamma$ and \\
    	$\foptpm = \nicefrac{\eta}{\Gamma+1}$, 
    	i.e. $\foptam\geq \foptpm$ and $\foptgm\geq \foptpm$.}
    \label{fig:gapSAM_SPM}
  \end{minipage}    
  \caption{Instances in the proofs of Prop. \ref{Prop:gapSPM_SAM} 
  	({\footnotesize A}) and Prop. \ref{Prop:gapSAM_SPM} ({\footnotesize B}).}
\end{figure}

\begin{prop}\label{Prop:gapSAM_SPM}
	Let $\Gamma\in\N$. For any $\alpha = \Gamma+1-\nicefrac{1}{\beta}$ with 
	$\beta\in \N$ there are instances such that ${\foptam=\alpha \foptpm}$ and 
	${\foptgm=\alpha \foptpm}$.
\end{prop}

\begin{proof}
  
  For an arbitrary $\Gamma\geq 1$ we construct an instance $I=(G,\Gamma)$ 
  that consists of three nodes $\{s, v, t\}$. Let 
  \begin{equation}
  	\eta:=\frac{\Gamma^2+\Gamma}{\Gamma+1-\alpha}=\beta(\Gamma^2+\Gamma)\in\N.
  \end{equation}
  There are $\Gamma+1$ parallel arcs $a'_i,~i\in[\Gamma+1]$, from $s$ to $v$ 
  with capacity $\eta$.
  The node $v$ is connected to $t$ by $\eta$ parallel arcs $a''_i, ~i\in[\eta]$, 
  with unit capacity, cf. \Cref{fig:gapSAM_SPM}.
  
  First, we show that an optimal solution $\xoptam$ to \SAM ~on $I$ has 
  objective value ${\foptam = \eta-\Gamma}$. 
  We may assume w.l.o.g. that $\xoptam(a'_i)=u_{a'_i}=\eta$ for all $i\in[\Gamma+1]$.
  Due to robust flow conservation \eqref{static:edge-rob-constraint}, 
  at most $\eta$ units of flow may be sent from node $v$ to $t$. 
  Since $\xoptam$ is optimal, these $\eta$ flow units are distributed equally along the 
  parallel arcs from $v$ to $t$, fully utilizing their capacity. 
  Since $\eta\geq\Gamma+1$, $\Gamma$ many arcs from $v$ to $t$ with unit 
  capacity are deleted in the worst case, resulting in an optimal value 
  $\foptam = \eta-\Gamma$. 
  
  Next, we show that an optimal solution $\xoptpm$ to \SPM ~has 
  objective value $\foptpm = \nicefrac{\eta}{(\Gamma+1)}$.
  Obviously, $\xoptpm$ saturates each of the parallel arcs from
  $v$ to $t$ and thus $\sum_{P\in\mathcal{P}}\xoptpm(P)=\eta$.
  As the optimal flow minimizes the flow that is deleted in the worst case, 
  it is uniformly distributed on the $\Gamma+1$ parallel arcs from $s$ to $v$. 
  Thus, $\sum_{P\in\mathcal{P}:a'\in P}x(P)=\nicefrac{\eta}{(\Gamma+1)}$ 
  for all arcs $a'$ from $s$ to $v$.
  Since $\eta\geq \Gamma+1$, $\Gamma$ of the arcs connecting $s$ and $v$ are 
  deleted in the worst case and the optimal value is given by 
  $\foptpm=\eta-\nicefrac{\Gamma\eta}{(\Gamma+1)}=\nicefrac{\eta}{(\Gamma+1)}$. 
  
  Combining the above yields 
  \begin{align*}
    \frac{\foptam}{\foptpm} = 
    \frac{\eta-\Gamma}{\nicefrac{\eta}{(\Gamma+1)}} = 
    \frac{(\Gamma+1)(\eta-\Gamma)}{\eta} = 
    \Gamma+1 - \frac{\Gamma^2+\Gamma}{\eta} = 
    \alpha.
  \end{align*} 
  
  We note that the optimal solution $\xoptam$ to \SAM ~is also optimal 
  to \SGM ~so that ${\foptgm=\alpha \foptpm}$.
\end{proof}

For the special case of $\Gamma=1$ and instances without directed cycles, 
i.e. directed acyclic graphs (DAGs), we prove tightness of this bound. 
We thereby use the following lemma.
\begin{lemma}\label{lemma:reducedflowproperty}
  Let $\Gamma=1$ and let $G$ be a DAG. 
  For any feasible solution $\tildex$ to \SGM ~there exists a feasible solution~
  $x$ to \SGM ~with 
  \begin{enumerate}
  	\item[\textit{(i)}] ${\sum_{P:a\in \barP}x(P)\leq \sum_{P:a\in \barP}\tildex(P)}$
  	for all ${a\in A}$,
  	\item[\textit{(ii)}] ${x(P)=\tildex(P)}$ for all ${P\in\dinPt}$, implying
  	  ${\fnom(x)=\fnom(\tildex)}$ and ${f_R(x)=f_R(\tildex)}$,
  \end{enumerate}
  and
  \begin{align}\label{eq:reducedflowproperty}
  	\sum_{\substack{P\in\dinP(v):\\a\in P}}x(P)\leq \sum_{P\in\doutP(v)}x(P)
  	\quad\quad \forall a\in A
  \end{align}
  at every node $v\in V\setminus\{s,t\}$. 
\end{lemma}
For the proof, we refer to \Cref{prooflemma:reducedflowproperty}.

The following proof of tightness starts with an optimal solution to \SGM ~and then 
explicitly constructs a solution to \SPM ~on the same instance. 
We show that this construction, namely the deletion of flow on all subpaths, 
along with a suitable shifting of flow, does not reduce the robust flow value by 
more than half. 

\begin{theorem}\label{Prop:gapSGMSPMupperboundGamma1}
  Let $\Gamma=1$ and let $G$ be a DAG. 
  Then, $\foptgm\leq 2\foptpm$ and $\foptam\leq 2\foptpm$.
\end{theorem}
\begin{proof}
  For a feasible solution $\xgm$ to \SGM ~and an arc $a$,
  we denote the amount of flow on paths ending at $t$ and using  arc $a$ by 
  \begin{align*}
    \mu(\xgm,a)\coloneqq\sum_{\substack{P\in\dinPt:\\a\in P}}\xgm(P).
  \end{align*}
  We then have, due to $\Gamma=1$, 
  \begin{align}\label{eq:rob_value_mu}
    f_R(\xgm)=\fnom(\xgm)-\max_{a\in A}\mu(\xgm,a).
  \end{align}
  Let $\xoptgm$ be an optimal solution to \SGM ~which satisfies 
  \eqref{eq:reducedflowproperty}. 
  The existence of such a solution follows from \Cref{lemma:reducedflowproperty}. 
  
  We construct a feasible solution $\xpm$ to \SPM ~on $G$, for which we will show 
  that $f_R(\xpm)\geq\nicefrac{1}{2}\foptgm$.
  
  The construction is an iterative procedure, starting with $x_0=\xoptgm$.
  The procedure iterates over incoming paths of $t$ and in each iteration 
  $i=1,2,\dots$ of the procedure, we will construct a new flow $x_i$ as follows:
  let $P_i\in\dinPt$ be a $v_i$-$t$-path with $v_i\neq s$, $x_{i-1}(P_i)>0$ 
  and minimum number of arcs.
  If no such path exists, we stop the procedure as we will discuss later.
  We first construct an interim solution $\tilde{x}_i$.
  For this, we shift all the flow from $P_i$ to 
  incoming paths of $t$ that contain $P_i$ as subpath, i.e. to paths in the set
  \begin{align*}
    \Pa_i:=\{\bar{P}\in\dinPt\defsep 
      \bar{P}=P'\cup P_i\text{ with }P'\in\dinP(v_i)\}.
  \end{align*} 
  In detail, we set $\tilde{x}_i(P_i):= 0$ and
  \begin{alignat}{2}
    \tilde{x}_i(\bar{P})&:=x_{i-1} (\bar{P})
      +&&\frac{x_{i-1}(P')}{\sum_{P\in\dinP(v_i)}x_{i-1}(P)}\cdot x_{i-1}(P_i)
      \quad\forall\bar{P}=P'\cup P_i\in\Pa_i, 
      \label{Equ:Upperbound_construction_xi}\\
    \tilde{x}_i(P')&:=x_{i-1}(P')
      -&&\frac{x_{i-1}(P')}{\sum_{P\in\dinP(v_i)}x_{i-1}(P)} \cdot x_{i-1}(P_i)
      \quad\forall P'\in\dinP(v_i),\\
    \tilde{x}_i(P)&:=x_{i-1}(P) \quad
      &&\forall P\in\barP\setminus(\{P_i\}\cup\Pa_i\cup \dinP(v_i)).
      \label{Equ:Upperbound_construction_xi2}
  \end{alignat}
  By this, we reduce the flow on paths ending at $v_i$, namely
  in total by $x_{i-1}(P_i)$. 
  The outgoing flow of $v_i$ is also reduced by $x_{i-1}(P_i)$ on $P_i$. 
  Thus, $\tilde{x}_i$ still satisfies robust flow conservation
  \eqref{static:general-rob-constraint} at $v_i$ and is therefore feasible 
  for \SGM ~since also the capacity constraints remain satisfied.
  
  From $\tilde{x}_i$, a feasible solution $x_i$ 
  satisfying (i), (ii) (cf. \Cref{lemma:reducedflowproperty}), 
  and \eqref{eq:reducedflowproperty} can be constructed due 
  to \Cref{lemma:reducedflowproperty}.

  In some iteration $K$, there exists no path
  $P_K\in\dinPt$ with $P_K\notin\doutP(s)$ and $x_{K-1}(P_K)>0$.
  We define $\xpm(P):=x_{K-1}(P)$ for all $P\in\Pa$ and $\xpm(P):=0$ for all 
  $P\in\barP\setminus\Pa$ and stop with the procedure.
  Note that $\xpm$ is a feasible solution to \SPM ~as there is no flow on subpaths.
  
  Furthermore, this procedure ends after finitely many iterations, 
  i.e. $K<\infty$, as we consider every path ending at $t$ at most once.
  
  We now prove two further properties of $\xpm$.\\
  First, in every iteration $i\in [K-1]$, the inflow of $t$ does not change, 
  since, due to (ii), 
  \begin{align*}
    \sum_{\bar{P}\in\Pa_i}x_i(\bar{P})+x_i(P_i)
      =\sum_{\bar{P}\in\Pa_i}x_i(\bar{P})
      =\sum_{\bar{P}\in\Pa_i}x_{i-1}(\bar{P})+x_{i-1}(P_i),
  \end{align*}
  which implies equality of the nominal values $\fnom(\xpm)=\fnom(\xoptgm)$.
  
  Second, we investigate the amount of flow on the single arcs. 
  Let $a'\in A$ be an arbitrary arc. 
  
  For every iteration $i\in [K-1]$, $x_{i-1}$ satisfies robust 
  flow conservation \eqref{static:general-rob-constraint} at $v_i$, 
  which in particular due to $\Gamma=1$ implies
  \begin{align*}
    \sum_{P\in\dinP(v_i):a'\notin P}x_{i-1}(P)\geq 
      \sum_{P\in\doutP(v_i)}x_{i-1}(P).
  \end{align*}   
  Together with \eqref{eq:reducedflowproperty} we obtain
  \begin{align*}
    \frac{\sum_{P\in\dinP(v_i):a'\in P}x_{i-1}(P)}
      {\sum_{P\in\dinP(v_i)}x_{i-1}(P)}\leq \frac{1}{2}.
  \end{align*}  
  Therefore, by summing up \eqref{Equ:Upperbound_construction_xi} 
  and \eqref{Equ:Upperbound_construction_xi2} for all paths 
  that end at $t$ and contain $a'$, it follows that
  \begin{align*}
    \sum_{\mathclap{\substack{P\in\dinPt:\\a'\in P}}}~~~x_i(P) 
      =~~~ \sum_{\mathclap{\substack{P\in\dinPt:\\a'\in P}}}~~~x_{i-1}(P) 
      	+ \frac{\sum_{P\in\dinP(v_i):a'\in P}x_{i-1}(P)}
      		{\sum_{P\in\dinP(v_i)}x_{i-1}(P)} x_{i-1}(P_i)
      \leq ~~~\sum_{\mathclap{\substack{P\in\dinPt:\\a'\in P}}}~~~x_{i-1}(P)
      	+  \frac{1}{2}x_{i-1}(P_i).
  \end{align*}
  We recall that in every iteration $i$, we add flow on 
  incoming paths of $t$ of amount $x_{i-1}(P_i)$ in total. 
  The above inequality hence says that at most half of this amount of flow 
  uses the arc $a'$. 
  Thus, for $\xpm$ and the arc $a' \in A$, we have 
  \begin{align*}
    \sum_{\substack{P\in\Pa:\\a'\in P}}\xpm(P)
      \leq \mu(\xoptgm,a') 
      + \sum_{\substack{P\in\Pa:\\a'\notin P}}\xpm(P)
      \leq \max_{a\in A}\mu(\xoptgm,a) 
      + \sum_{\substack{P\in\Pa:\\a'\notin P}}\xpm(P).
  \end{align*}

  Hence, 
  \begin{align*}
    \sum_{\substack{P\in\Pa:\\a'\notin P}}\xpm(P)\geq 
      \sum_{\substack{P\in\Pa:\\a'\in P}}\xpm(P) - \max_{a\in A}\mu(\xoptgm,a).
  \end{align*}
  Putting together the above, we obtain
  \begingroup
  \allowdisplaybreaks
  \begin{align*}
    \foptgm &= \fnom(\xoptgm)-\max_{a\in A}\mu(\xoptgm,a)
    =\fnom(\xpm)-\max_{a\in A}\mu(\xoptgm,a)\\
    &=\sum_{\substack{P\in\Pa:\\a'\notin P}}\xpm(P)
      +\sum_{\substack{P\in\Pa:\\a'\in P}}\xpm(P)
      -\max_{a\in A}\mu(\xoptgm,a)
    \geq 2\sum_{\substack{P\in\Pa:\\a'\in P}}\xpm(P)
      - 2\max_{a\in A}\mu(\xoptgm,a),
  \end{align*}
  \endgroup
  which is equivalent to
    $\sum_{P\in\Pa:a'\in P}\xpm(P)
      \leq \max_{a\in A}\mu(\xoptgm,a) + \frac{1}{2}\foptgm$.
  
  Since $a'\in A$ has been chosen arbitrarily, we can conclude the proof of the 
  first claim by
  \begin{align*}
    \foptpm\geq f_R(\xpm)&=\fnom(\xpm)
      -\max_{a\in A}\sum_{\substack{P\in\Pa:\\a\in P}}\xpm(P)
    =\fnom(\xoptgm)-\max_{a\in A}\sum_{\substack{P\in\Pa:\\a\in P}}\xpm(P)\\
    &\geq \fnom(\xoptgm) - \max_{a\in A}\mu(\xoptgm,a) - \frac{1}{2}\foptgm
    ~= \frac{1}{2}\foptgm.
  \end{align*}
  
  Moreover, every optimal solution to \SAM ~is feasible for \SGM~and thus 
  ${\frac{1}{2}\foptam\leq \frac{1}{2}\foptgm\leq \foptpm}$. 
\end{proof}

We recall that both, the path model and the general model, are solvable in polynomial time for $\Gamma=1$. Thus, the former results show that \SGM ~yields 
robust optimal values up to twice as large as the ones of \SPM ~in this computationally tractable special case.

Note that the proof of \Cref{lemma:reducedflowproperty} 
requires the graph to be acyclic. If this lemma can be extended
to general directed graphs, \Cref{Prop:gapSGMSPMupperboundGamma1} can be extended to those as well.

Due to the simple structure of the instances in \Cref{Prop:gapSAM_SPM} and tightness on DAGs with $\Gamma=1$,
we suppose that the lower bounds on the gap from \Cref{Prop:gapSAM_SPM} are tight in general.  
We therefore state the following conjecture.

\begin{conj}\label{conj:upperbound}
  For any $\Gamma\in\N$ the inequalities $\foptgm\leq (\Gamma+1)\foptpm$ and 
  $\foptam\leq (\Gamma+1)\foptpm$ hold on every instance.
\end{conj}

The proofs of \Cref{lemma:reducedflowproperty}  as well as
\Cref{Prop:gapSGMSPMupperboundGamma1} heavily 
exploit the fact that in the case of ${\Gamma=1}$ it is easy to determine the arc that fails in the worst case, as it is the arc that maximizes $\mu$. 
Thus, for ${\Gamma>1}$, a different proof technique seems to be necessary for \Cref{conj:upperbound}.


\subsection*{Price of Robustness}\bookmarksetup{depth=part}

We conclude this section by comparing the nominal value of the robust
optimal solutions to the nominal optimal value~$\foptnom$. 
We hence analyze the loss that may arise when hedging 
against uncertainties that do not realize
which has been introduced as the \textit{Price of Robustness (PoR)} in \cite{PoR}. 
To the best of our knowledge, there are no earlier results on the PoR in the robust network flow context. 

We start by a result on the PoR for the arc model, for which we use the instance from the proof of  \Cref{Prop:gapSPM_SAM}, cf. \Cref{fig:gapSPM_SAM}. 
Note that the nominal value of all robust feasible solutions 
to \SAM ~in this instance is zero, while the nominal 
optimal flow has a value greater than zero. 
We infer that, in general, the PoR can be unbounded: 
\begin{corollary}\label{prop:gapNom_SAMNom}
    For any $\Gamma\in\N$ and $\alpha\in\R$ there are instances such that 
    $\foptnom>\alpha \fnom(\xam)$ for all robust feasible solutions 
    $\xam$ to \SAM.
\end{corollary}

In contrast, for $\Gamma=1$ there exists an optimal solution $\xoptpm$ to \SPM ~with 
$\foptpmnom = \foptnom$ and an optimal solution $\xoptgm$ to \SGM ~with 
$\foptgmnom = \foptnom$, as shown in \cite{aneja} and 
\Cref{lemma:static:gamma1-maxflow}, respectively.
Robustness thus can be achieved without loss in the nominal flow value if $\Gamma=1$. 
However, this is not the case for $\Gamma\geq 2$ anymore.

\begin{prop}\label{Prop:gapNom_SPMNom}
  Let $\Gamma\geq 2$ and $\alpha \in [1,\nicefrac{2\Gamma}{\Gamma+1})\cap \Q$. 
  Then, there are instances such that 
  ${\foptnom= \alpha \foptpmnom=\alpha \fnom(\xoptgm)}$, where 
  $\xoptpm$ is an optimal solution to \SPM ~with maximal nominal value and
  $\xoptgm$ is an optimal solution to \SGM ~with maximal nominal value.
\end{prop}
In the proof of \Cref{Prop:gapNom_SPMNom} in \Cref{proofProp:gapNom_SPMNom}, 
we provide such a class of instances.\\
Going beyond our results on the PoR for the three static models, some open 
questions remain, e.g., if there is an upper bound on the PoR 
for \SPM ~and \SGM, or if there exist instances on which the PoR for \SPM ~is 
higher than for \SGM.

%% file: dynamic.tex
\bookmarksetup{depth=subsection}
\section{Robust dynamic maximum flows}\label{sec:dynamic}
We now extend the concepts of robust maximum flows to a dynamic setting. 
In detail, we consider maximum flow problems in which a travel time is assigned to each arc. 
These travel times then are affected by uncertainty in the sense that delays may occur. 
We model such problems as follows.
A graph $G=(V,A,s,t,u,\tau,\Delta\tau)$ is equipped with
travel times $\tau_a\in\N_0$ and delays $\Delta\tau_a\in\N_0$ on each arc $a\in A$. 
Therefore, in figures throughout this section, arc labels denote the name of the arc, the travel time, the delay and the capacity as shown below.
\begin{figure}[h]
		\begin{tikzpicture}
    	\node(v) at (0,0) [circle,draw=black, thick, minimum size=6mm] {$v$};
    	\node(w) at (4,0) [circle,draw=black, thick, minimum size=6mm] {$w$};
        \draw[thick,->](v) edge node[above]{$(a,\tau_a,\Delta\tau_a,u_a)$}(w);
        \end{tikzpicture}
        \label{fig:arclabelsDynamic}
\end{figure}

If an amount of flow enters an arc $a=(v,w)\in A$ at a point in time, denoted by $\theta$, then it arrives at $w$ at time $\theta+\tau_a$. 
Given a time horizon $T\in\N$, we define the set of time steps until the 
time horizon by $\T:=[T]=\{1,\dots,T\}$. 
We say that an amount of flow arrives at the sink within the time horizon if it reaches $t$ at any point in time $\theta\in\T$.
Flow solutions are again defined on arcs or paths and now also depend on time, as we have to specify when a certain amount of flow starts to travel along an arc or a path. 
$x(a,\theta)$ describes the inflow rate of arc $a$ at time $\theta$ and 
$x(P,\theta)$ the inflow rate of path $P$ at time $\theta$, respectively.
Furthermore, we consider uncertainties in the form that at most $\Gamma$ many arcs may be delayed.
The set of all possible scenarios $\Sc$ and its corresponding 
indicator set $\Lambda$ as well as the set of all simple \stpaths ~$\Pa$ and the set of all subpaths $\barP$ are defined as stated in \Cref{sec:Definitions}. 
For simplicity, we allow flow only on simple paths. 
Sending flow on cycles would introduce additional complexity, as then there would  be infinitely many paths.
Such a problem in turn could then be modeled alternatively by allowing for waiting. 
The discussion of such and other extensions of our model is beyond the scope of this work.
\\ 

The nominal dynamic maximum flow problem maximizes the flow arriving at the sink $t$ within the time horizon $T$ while satisfying flow conservation as well as capacity constraints. 
We set $x(a,\theta):=0$ for all $a\in A, ~\theta\in\Z\setminus{\T}$. 
Then, a simple discrete nominal dynamic maximum flow model using strict flow conservation is given by
\begin{subequations}\label{dynamic:nominal}
    \begin{alignat}{2}
    f^*=~\max_{x\in\R_{\geq 0}^{|A|T}} \quad
        &\sum_{\mathclap{a\in\dinA(t)}}\quad
         \sum_{\theta\in [T-\tau_a]}x(a,\theta)\\[5pt]
    \st \quad 
    	&\sum_{\mathclap{a\in\dinA(v)}}\quad x(a,\theta-\tau_a)
    		=\sum_{a\in\doutA(v)}x(a,\theta), \qquad 
        	&&	\forall v\in V\setminus\{s,t\},\theta\in \T, \\
        &x(a,\theta)\leq u_a,  &&\forall a\in A,\theta\in \T.
	\end{alignat}
\end{subequations}
There are several variants of this problem in the literature, for example continuous models, models using weak flow conservation and models which allow to store flow at intermediate nodes, cf. \cite{skutella}. 
In \cite{tardos}, the authors point out  a close correspondence between discrete
and continuous nominal dynamic flows.

\subsection{Robust dynamic flow models}\label{sec:dynamic:models}
In the dynamic setting, flow requires a certain travel time $\tau$ to traverse an arc. 
If an arc $a$ is delayed, the travel time on this arc changes from $\tau_a$ to $\tau_a+\Delta\tau_a$ for all flow sent on arc $a$ at any point in time.

When defining flow on paths, $\tau_P=\sum_{a\in P}\tau_a$ describes the travel time of path $P$ 
and $\Delta_z(P)\in\N_0$ the delay of path $P$ in scenario $z$.

For a path $P$ and an arc $a \in P$, we denote the point in time when flow, which enters arc $a$ at time $\theta$ in scenario $z$ via path $P$, has entered path $P$ at its first node by $\theta_{z,a}(P)$ (cf. \cite{gottschalk}): 
 \begin{equation*}
\theta_{z,a}(P):=\theta-\sum_{a'\in A:a'\prec_P a}(\tau_a'+z(a')\Delta\tau_{a'}),
 \end{equation*}
where $a'\prec_P a$ means that $a', a\in P$ and  $a'$ is visited before $a$.

Additionally, by $\theta_z(P)$ (and $\theta_z(a)$) we denote the point in time when flow, which arrives at the end node of path $P$ (or arc $a$) at time $\theta$ in scenario $z$, has entered path $P$ (or arc $a$) at its first node:
\begin{equation*}
	\theta_z(P):= \theta - (\tau_P + \Delta_z(P)) \quad\text{and}\quad
	\theta_z(a):= \theta - (\tau_a + z(a)\Delta\tau_a).
\end{equation*}

\subsection*{Dynamic path model (DPM)}\bookmarksetup{depth=part}
To the best of our knowledge, there is only one model for the robust dynamic maximum flow problem with delays in the literature so far, cf. \cite{gottschalk}. 
This model is an extension of the path model \SPM ~to the dynamic case, i.e., flow is assigned to \stpaths ~again.
\cite{gottschalk} set up the model in a continuous framework, where the flow is then a function of the time. 
They show that there always exists a piecewise constant flow solution that changes its values only at integer points in time and still solves the continuous model to optimality. 
For this reason, we give their model, which we call dynamic path model \DPM, in a discretized version.
We define $x(P,\theta):=0$ for all $P\in \Pa, ~\theta \in \Z\setminus \T$. 
Then, the dynamic path model \DPM ~reads
\begin{taggedsubequations}{DPM}\label{dynamic:path}
    \begin{alignat}{2}
    \foptdpm:= 
    \max_{\mathclap{x \in \R_{\geq 0}^{|\Pa|T}}} \quad
    	&\min_{z\in\Lambda} \quad \sum_{P\in\Pa}  \qquad \quad
    		\sum_{\mathclap{\theta\in[T-\tau_P-\Delta_z(P)]}} \qquad x(P,\theta) 
    		\qquad &&\\[5pt]
    \st \quad &\sum_{\mathclap{P\in\Pa:a\in P}} \quad x(P,\theta_{z,a}(P))\leq u_a, 
        		\quad &&\forall a\in A, \theta\in\T, z\in\Lambda.
    \end{alignat}
\end{taggedsubequations}
We denote an optimal solution to \DPM ~by $\xoptdpm$ and the optimal value 
by $\foptdpm$.

An important solution concept for the nominal dynamic maximum flow problem are temporally repeated flows, 
cf. e.g. \cite{skutella}.
A feasible solution is called temporally repeated if the inflow rate 
$x(P,\theta)$ of all paths $P\in\Pa$ is constant over all $\theta\in\T$.
While there always exists a temporally repeated flow solving the nominal 
problem to optimality, there is not necessarily a temporally repeated flow that solves \DPM ~to optimality, cf. \cite{gottschalk}.

\subsection*{Dynamic arc model (DAM)}
We now introduce the dynamic arc model \DAM, which is the natural extension of the 
static arc model \SAM ~to the dynamic case, i.e. the flow solution has one entry for
each arc and for each time step. 
In the static setting, robust flow conservation ensured that the inflow rate at 
each node is higher than the outflow rate, no matter which scenario realizes. 
We extend this constraint to the dynamic setting by requiring 
robust flow conservation for each point in time.
We define $x(a,\theta):=0$ for all $a\in A, ~\theta\in\Z\setminus{\T}$. 
Then, the dynamic arc model \DAM ~is given by 
\begin{taggedsubequations}{DAM}\label{dynamic:arc}
    \begin{alignat}{2}
    \foptdam:= 
	\max_{\mathclap{x\in\R_{\geq 0}^{|A|T}}} \quad
		&\min_{z\in\Lambda} \quad \sum_{\mathclap{a\in\dinA(t)}} \qquad \qquad
        \sum_{\mathclap{\theta\in [T-\tau_a-z(a)\Delta\tau_a]}} \qquad 
        x(a,\theta) \\[3pt] 
    \st \quad &\sum_{\mathclap{a\in\dinA(v)}} ~~ x(a,\theta_z(a)) \geq ~~~
    	 \sum_{\mathclap{a\in\doutA(v)}} ~~ 
    	x(a,\theta), ~~\qquad
        &&\forall v\in V \setminus\{s,t\}, \theta\in \T, z\in\Lambda,
        	\label{dynamic:arc:rob-constraint}\\
    	& x(a,\theta) \leq u_a,
    	&&\forall a\in A,\theta\in \T.
    \end{alignat}
\end{taggedsubequations}
We denote an optimal solution to \DAM ~by $\xoptdam$ and the optimal value 
by $\foptdam$.

\subsection*{Dynamic general model (DGM)}
Finally, we also extend the new general model to the dynamic setting, i.e. flow is assigned to subpaths. 
Analogously to the other two dynamic models, we require flow conservation and capacity constraints to hold for every possible scenario at each point in time.
We define $x(P,\theta):=0$ for all $P\in\barP, ~\theta\in\Z\setminus{\T}$.
The dynamic general model \DGM ~then reads
\begin{taggedsubequations}{DGM}\label{dynamic:general}
    \begin{alignat}{4}
    \foptdgm:= 
    \max_{\mathclap{x\in \R_{\geq 0}^{|\barP|T}}} \quad
    	&\min_{z\in\Lambda} \quad
    	&&\sum_{\mathclap{P\in\dinPt}}  \qquad \qquad
    	&&\sum_{\mathclap{\theta\in [T-\tau_P-\Delta_z(P)]}} ~~\quad x(P,\theta) 
    	\qquad && \label{dynamic:general-obj}\\[5pt]
    \st \quad 
    	&\sum_{\mathclap{P\in\dinP(v)}} 
    	&&x(P,\theta_z(P))
    	&&\geq ~~~ \sum_{\mathclap{P\in\doutP(v)}} \quad x(P,\theta), \quad 
    	&&\forall v\in V\setminus \{s,t\},\theta\in\T, z\in\Lambda,
    		\label{dynamic:general-rob-constraint}\\
     	&\sum_{\mathclap{P\in\barP:a\in P}}
    	&&x(P,\theta_{z,a}(P))
    	&&\leq ~u_a, \quad
    	&&\forall a\in A,\theta\in\T,z\in\Lambda.
    \end{alignat}
\end{taggedsubequations}
\addtocounter{equation}{-2}

Given a feasible solution $x$ to \eqref{dynamic:general}, 
we denote its objective value by $f_R(x)$. 
We denote the optimal solution usually by $\xoptdgm$ and the optimal value by  
$\foptdgm=f_R(\xoptdgm)$. 

As in the static case, we observe that \DPM ~and \DAM ~are special cases of \DGM 
~where the solutions are restricted to \stpaths ~or paths consisting of 
single arcs, respectively. 
Therefore, any feasible solution to \DPM ~or \DAM ~can be extended such that 
it is feasible for \DGM ~by defining the flow on the subpaths, which are not
\stpaths ~or single arcs, respectively, to be equal to zero. 
For ease of notation, we write ${\xdpm\in\R^{|\Pa|}}$ 
when talking about \eqref{dynamic:path} and ${\xdpm\in\R^{|\barP|}}$ when talking 
about \eqref{dynamic:general}, which implies $\foptdpm=f_R(\xoptdpm)$.
Similarly, we proceed with feasible solutions to \DAM.

Additionally, we remark that the static robust maximum flow models are special cases of their respective dynamic counterparts. 
Given an instance for the static problem, we can transform it to an instance 
for the dynamic problem by defining $\tau\equiv 0$, $\Delta\tau\equiv 1$, 
and $T=1$.
The delay of any arc $a$ increases the travel time so that flow does not arrive at 
the sink within the time horizon if it uses $a$. 
Therefore, the delay of an arc has the same effect as if this arc would fail. 
Since the travel times are zero and the time horizon is one, an optimal solution
to one of the dynamic model thus maximizes the objective value of the corresponding 
static model.

\subsection*{Continuous models} In the above dynamic flow models, we consider discretized flow solutions. 
The discretization has been made over the time, so that the flow that is assigned to a path and a point in time can only be assigned to integer points in time. 
We justify this restriction of our dynamic models to the discretized settings as follows: 
one can show that, even if one allows for arbitrary Lebesgue-integrable functions over time as flow solutions, the models can be solved to optimality by piecewise constant functions that change their values only at integer points in time. 
This has been shown for the dynamic path model in \cite[Proposition 1]{gottschalk}. 
Here, we show the statement for the dynamic general model and the dynamic arc model.
In the following, we denote by continuous dynamic model a model that allows, in contrast to the stated ones, for any Lebesgue-integrable function over time as flow solution. 
A formal model for such a continuous framework can be easily derived by slightly modifying the previously stated ones: the domain of $x$ changes as described before and 
all sums over time are replaced by integrals over time.

\begin{prop}\label{lem:DAM/DGM:pw-constant} 
The continuous general model and the continuous arc model can be solved to optimality by a piecewise 
constant function: 
	\begin{enumerate}
	\item[\textit{(i)}] There is a piecewise constant function which changes 
		its values only at integer points and solves the dynamic general model 
		to optimality.
	\item[\textit{(ii)}] There is a piecewise constant function which changes 
		its values only at integer points and solves the dynamic arc model 
		to optimality.
	\end{enumerate}
	\end{prop}
The proof of \Cref{lem:DAM/DGM:pw-constant} is similar to the proof of the analogous statement for
the dynamic path model from \cite{gottschalk} and can be found in 
\Cref{prooflem:DAM/DGM:pw-constant}.


\bookmarksetup{depth=subsection}
\subsection{Complexity of the dynamic models}\label{sec:dynamic:complexity}

In the following, we provide an overview of the computational complexity of the 
three robust dynamic maximum flow models. 
We start with the complexity of \DPM: 
\cite{gottschalk} prove that solving \DPM ~is at least as hard as \SPM, 
which implies NP-hardness for $\Gamma$ being part of the input. 
They further investigate the complexity of computing a temporally repeated flow 
and provide several results. We supplement the results from \cite{gottschalk} 
by a complexity proof for fixed $\Gamma\geq 1$, which they have left as an open 
question. 
Subsequently, we show that \DAM ~is solvable in polynomial time followed by 
various results concerning the computational complexity of \DGM.
A summary of the complexity results can be found in \Cref{tab:ComplexityOverviewDynamic}.
\begin{table}[h]
\centering
\resizebox{1\textwidth}{!}{
\begin{tabular}{l|l|l|l|l}
        & Arc Model \DAM & \multicolumn{2}{c|}{Path Model \DPM} 
        & General Model \DGM \\[3pt]
        & & general solutions & temporally repeated &
    \\\hline	& & & &\\
    fixed $\Gamma\geq 1$ 
    	& polynomial time 
    	& NP-hard  
    	& NP-hard 
    	& NP-hard \\
    	& (\Cref{thm:DAM:reform}) 
    	& (\Cref{thm:DPM:gamma1})
    	& (\Cref{cor:DPM:fixedgamma-temprep})
    	& (\Cref{thm:DGM:fixedgamma})
    \\	& & & &\\[1pt]
    $\Gamma$ arb.
    	& polynomial time 
    	& NP-hard 
    	& NP-hard 
    	& strongly NP-hard \\    	
    	& (\Cref{thm:DAM:reform}) 
    	& \cite{gottschalk} 
    	& \cite{gottschalk} 
    	& (\Cref{thm:DGM:gammainputNP})
\end{tabular}}
    \caption{Overview of the complexity results of the 
    			robust dynamic flow models.}
    \label{tab:ComplexityOverviewDynamic}
\end{table}
\subsection*{Complexity of \DPM}\bookmarksetup{depth=part}
\begin{theorem}\label{thm:DPM:gamma1}
  \DPM ~is NP-hard for fixed $\Gamma\geq 1$.
\end{theorem}
\begin{proof}
  We first show that \DPM ~is NP-hard for $\Gamma=1$ and then 
  extend this result to fixed $\Gamma>1$. 
  The proof for $\Gamma=1$ closely follows the proof by which \cite{bertsimas} 
  show weakly NP-hardness of the adaptive maximum flow problem under arc 
  failure and which in turn is based on \cite[Theorem 1, proof]{melkonian}. 
  Indeed, the delays we use in the following are sufficiently large such 
  that they correspond to the failure of an arc.
  
  We prove the statement by a reduction from the weakly 
  NP-complete \textsc{partition} problem:
  let $n\in\mathbb{N}$ and given a set of integers $b_i \in \mathbb{N}$, 
  $i\in[n]$, with $\sum_{i\in [n]} b_i=2L$ for some $L\in\N$, 
  is there a subset $I\subset [n]$ such that $\sum_{i \in I} b_i = L$?
  
  We construct an instance $(G=(V,A,s,t,u,\tau,\Delta\tau),\Gamma,T)$ 
  of \DPM ~from such a \textsc{partition} instance as follows,
  cf. \Cref{fig: Network NP-proof}.
  Let $\Gamma=1$, $\bar{b}=\max_{i\in[n]}b_i$ and $T=(2 n \bar{b} +1)L +1$.
  The node set is given by $V=\{s=v_1, v_2,...,v_n, v_{n+1}=t\}$.
  We connect for  every $i\in[n]$ the node $v_i$ to $v_{i+1}$, by two parallel 
  arcs $a_i^*$ and $a_i'$ with travel times $\tau_{a_i^*} := n \bar{b} b_i$ and
  $\tau_{a_i'} := (n \bar{b}+1) ~b_i = \tau_{a_i^*} + b_i$.
  We denote the set of arcs of type $a_i^*$ by $A^*$ and the set of arcs of type
  $a_i'$ by $A'$, so that $A=A^*\cup A'$. 
  We set $\Delta \tau_a:=T$ and $u_a:=1$ for all $a\in A$.
  \begin{figure}[h]
  \centering
  \begin{tikzpicture}[shorten >=2pt, node distance=3.3cm,on grid,auto, 
    roundnode/.style={circle, draw=black, thick, minimum size=8mm}]
    \node[roundnode] (v1) {$s$};
    \node[roundnode] (v2) [right = of v1] {$v_2$};
    \node[roundnode] (v3) [right = of v2] {$v_3$};
    \node[roundnode] (vn) [right = of v3] {$v_n$};
    \node[roundnode] (vn+1) [right = of vn] {$t$};
    \draw[thick,->,bend left=25](v1)edge node {$(a_1^*,n \bar{b} b_1,T,1)$}(v2);
    \draw[thick,->,bend right=25](v1) edge node[below]
    	{$(a_1',n \bar{b} b_1+b_1,T,1)$}(v2);
    \draw[thick,->,bend left=25](v2)edge node {$(a_2^*,n \bar{b} b_2,T,1)$}(v3);
    \draw[thick,->, bend right=25](v2) edge node[below] 
    	{$(a_2',n \bar{b} b_2+b_2,T,1)$}(v3);
    \draw[thick,->,dashed,bend left=25](v3) edge node {} (vn);
    \draw[thick,->,dashed,bend right=25](v3) edge node {} (vn);
    \draw[thick,->, bend left=25](vn) edge node 
    	{$(a_n^*,n \bar{b} b_n,T,1)$}(vn+1);
    \draw[thick,->, bend right=25](vn) edge node[below] 
    	{$(a_n',n \bar{b} b_n+b_n,T,1)$}(vn+1);
  \end{tikzpicture}    
	\caption{Instance constructed in the proof of \Cref{thm:DPM:gamma1}.
		This construction has been proposed in \cite{melkonian,bertsimas}.}
	\label{fig: Network NP-proof}
  \end{figure}
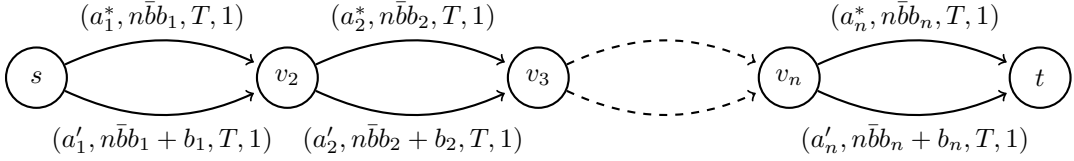
	
	We show that the \textsc{partition}  instance is a 'YES'-instance if and 
	only if
	$\foptdpm>0$ on the constructed instance of \DPM.
	
	First, assume that there exists a subset $I\subseteq S$ 
	such that $\sum_{i \in I} b_i = \sum_{i \in S\setminus I} b_i = L$.
	We define two arc-disjoint paths $P_1$ and $P_2$ by
	\begin{align*}
	  P_1 := \{a_i^*\in A^* \defsep i \in I \} ~ \cup ~ 
	    	\{a_i' \in A' \defsep i \in [n]\setminus I\}, \quad
	  P_2 := A\setminus P_1.
	\end{align*}    
	The travel times of these paths are given by
	\begin{align*}
	  \hspace{-0.5cm}
	  \tau_{P_1}  = \sum_{\substack{a_i^* \in A^*:\\i \in I}} \tau_{a_i^*} 
	    + \sum_{\substack{a_i' \in A':\\i \in [n]\setminus I}}\tau_{a_i'} 
      = n \bar{b}\sum_{i\in I}b_i+(n \bar{b} +1)\sum_{i \in [n]\setminus I} b_i 
      = (2n \bar{b}+1 ) L
	  = T-1,
	\end{align*}
	and analogously $\tau_{P_2}=T-1$.
	Let $x$ be the solution defined by $x(P_1,\theta)=x(P_2,\theta)=1$ for all 
	$\theta\in \T$.
	Without any arc being delayed, two units of flow arrive at the sink within 
	the time horizon $T$. 
	By delaying any arc $\tilde{a}$, the travel time of the path including 
	$\tilde{a}$ changes into $2T-1$, such that no flow arrives at the sink $t$ 
	via that path in time. 
	However, one unit of flow arrives at $t$ via the path not including 
	$\tilde{a}$ and thus $\foptdpm\geq f_R(x)=1>0$.
	
	Second, assume that $\foptdpm>0$.
	This implies that there exist two arc-disjoint paths $P_1$ and $P_2$ 
	with travel times at most $T-1$, such that the flow sent on either path 
	reaches the sink within $T$ if the other path is delayed. 
	We define
	$I := \left\{i \in [n] : a^*_i \in P_1\right\}$. 
	Assume for contradiction that 
	$\sum_{i \in I} b_i = L + \delta$ for some $\delta \neq 0$ 
	and thus $\sum_{i \in S\setminus I} b_i = L - \delta$. 
	Then, the travel times of $P_1$ and $P_2$ are given by
	\begin{align*}
	  	\tau_{P_1} & = (n \bar{b})\sum_{i\in I}b_i + (n \bar{b} +1)
	  		\sum_{i\in [n]\setminus I}b_i
	    \\ & = (n \bar{b})(L+\delta) +  (n \bar{b} +1)(L-\delta)
	       = (2n \bar{b} +1) ~L - \delta
	       = (T-1)-\delta
	\end{align*}
	and similarly,
	\begin{align*}
	  \tau_{P_2} & = (n \bar{b})\sum_{i\in [n]\setminus I} b_i 
	    + (n \bar{b}+1)\sum_{i\in I} b_i 
	         = (2n \bar{b} +1) ~L + \delta
	         = (T-1)+\delta.
	\end{align*}
	This contradicts the existence of two arc-disjoint paths 
	$P_1$ and $P_2$ with $\tau_{P_1}, \tau_{P_2} \leq (T-1)$. 
	Hence, $\sum_{i \in I} b_i = \sum_{i \in [n]\setminus I} b_i = L$, 
	which corresponds to a 'YES'-instance of the \textsc{partition} problem.
	Thus, computing an optimal solution to \DPM ~for $\Gamma=1$ is  
	NP-hard, which concludes the first part of the proof. \\	
	
	We now extend the result to any fixed $\Gamma\geq 1$.
	Given $\Gamma > 1$, we add $\Gamma-1$ parallel $s$-$t$-arcs to $G$. 
	Each of these additional arcs has capacity $u\equiv 2$, 
	travel time $\tau\equiv 0$ and delay $\Delta\tau\equiv T$, 
	i.e. the capacity of these arcs is larger than the capacity of any arc in 
	$G$ and the delay of an arc again corresponds to arc failure. 
	Let $\xoptdpm$ be an optimal solution to \DPM ~on this extended instance 
	and let $\lambda^*$ be the maximum amount of flow that can be deleted by delaying 
	$\Gamma$ arcs. By maximality of $\lambda^*$ and $\xoptdpm$, in the worst case, 
	all of the $\Gamma-1$ parallel $s$-$t$-arcs and one arc out of the original arcs of $G$ are delayed. 
	Hence, by solving \DPM ~on the extended instance with additional arcs for 
	some fixed $\Gamma$, we solve \DPM ~on $G$ for $\Gamma=1$ and 
	thus solve the according \textsc{partition} problem. 
	It follows that solving \DPM ~is NP-hard for any fixed $\Gamma\geq 1$.
\end{proof}

The optimal solutions to the instances in the proof of \Cref{thm:DPM:gamma1} 
are temporally repeated flows. 
\begin{corollary}\label{cor:DPM:fixedgamma-temprep}
  Computing a robust optimal temporally repeated flow to 
  \DPM ~is NP-hard for any fixed $\Gamma\geq 1$.
\end{corollary}

On the other hand, \cite{gottschalk} show that an optimal temporally repeated flow to \DPM ~can be computed in polynomial time on instances satisfying the 
T-bounded path length property. 
This property on an instance requires that the travel time including the possible delay of each \stpath ~is less than or equal to the time horizon $T$.

\subsection*{Complexity of \DAM}
The static version, \SAM, is solvable in polynomial time, as shown in \cite{bertsimas}. 
We prove a similar result for the dynamic case.
\begin{theorem}\label{thm:DAM:reform}
	A robust optimal solution $\xoptdam \in \R_{\geq0}^{|A|T}$ to \DAM ~can 
	be computed by solving the linear program
	\begin{subequations}\label{dynamic:arc-reform}
	\begin{alignat}{2}
    \max_{\mathclap{x,\eta, \lambda,\mu,\nu}} \quad
    	&\sum_{\mathclap{a\in\dinA(t)}} \hs{0.8em}
      	\sum_{\theta\in\T}~x(a,\theta-\tau_a)-
      	\sum_{a\in\dinA(t)}\nu(a)-\Gamma\mu\\[5pt]
    \st \quad
    	&\sum_{\mathclap{a\in\dinA(v)}}~~(x(a,\theta -\tau_a)
    		- \lambda_\theta(a))-\Gamma \eta_{v,\theta} 
    		\geq ~~\sum_{\mathclap{a\in\doutA(v)}} ~~x(a,\theta), \quad
    		&&\forall v\in V\setminus\{s,t\},\theta\in\T,\\
    	&\eta_{v,\theta}+\lambda_{\theta}(a)
    		\geq x(a,\theta-\tau_a)-x(a,\theta-\tau_a-\Delta \tau_a),\qquad
    		&&\forall v\in V\setminus \{s,t\},a\in\dinA(v), \theta \in \T, \\
    	&\mu+\nu(a)\geq \sum_{i\in [\Delta\tau_a]}x(a,T-\tau_a-(i-1)),
    		&&\forall a\in\dinA(t), \\
    	&x(a,\theta)\leq u_a, &&\forall a\in A,\theta\in\T, \\[3pt]
    	&x,\eta,\lambda,\mu,\nu \geq 0,&&
	\end{alignat}
	\end{subequations}
	with $\eta\in\R^{T(|V|-2)}$, 
	$\lambda\in\R^{T(|A|-|\dinA(t)|)}$, 
	$\mu \in \mathbb{R}$ , and
	$\nu \in \mathbb{R}^{|\dinA(t)|}$. 
	
	Hence, \DAM ~can be solved in time polynomial in the size of the input 
	$(G,\Gamma,T)$.
\end{theorem}
The proof utilizes standard techniques from robust optimization, such as dualizing the inner minimization problems, and can be found in \Cref{proofthm:DAM:reform}.

\subsection*{Complexity of \DGM}
For a feasible solution $x$ to \DPM ~or \DGM, we denote the first point in time at 
which flow reaches $t$ in any scenario, i.e. the point of earliest arrival, by $\thea(x)$:
\begin{align*}
  \thea(x)\coloneqq \min\{\theta\in\T\defsep 
    \forall z\in\Lambda~\exists P_z\in\dinPt: 
    x(P_z,\theta-\tau_{P_z}- \Delta_z(P_z))>0\}
\end{align*}
with $\thea(x):=\infty$ if no such $\theta\in\T$ exists.
Intuitively, it is obvious that considering \DGM ~compared to \DPM ~does not
speed up the flow since the only difference is that we allow flow to 
'seep away' at intermediate nodes in \DGM. 
We will use this idea as auxiliary lemma in the following and prove it formally.
\begin{lemma}\label{lemma:dynamic:earliestarrival}
Let $\xdgm$ be a feasible solution to \DGM. 
Then, there exists a feasible solution $\xdpm$ to \DPM ~on the same instance with 
$\thea(\xdpm)\leq \thea(\xdgm)$.
\end{lemma}
\begin{proof}
  We define a feasible solution $\xdpm$ to \DPM ~by
  \begin{align*}
    \xdpm(P,\theta):=c \quad\forall P\in\mathcal{P},~\theta\in\T,
  \end{align*}
  where $c>0$ is the maximum value such that $\xdpm$ is feasible.
  Such a $c>0$ exists since all capacities $u$ are positive.
  
  Let $\xdgm$ be a feasible solution to \DGM. 
  Let $z\in\Lambda$ be arbitrary and 
  ${P_z\in\dinP(t)\subseteq\barP}$ be a path on which flow of $\xdgm$ arrives 
  at the sink at the earliest time in scenario $z$.
  If $P_z$ is an \stpath, it follows directly that 
  $\xdpm(P_z,\theta)=c$ for all $\theta\in\T$ for which $\xdgm(P_z,\theta)>0$. 
  If $P_z$ is a $v$-$t$-path for any $v\in V\setminus\{s,t\}$, 
  there exists an $s$-$v$-path $P'_z$ with 
  \begin{align*}
    1+\tau_{P'_z}+ \Delta_z(P'_z))\leq \thea(\xdgm)-\tau_{P_z}- \Delta_z(P_z),
  \end{align*}  
  as otherwise $\xdgm$ would violate robust flow conservation  
  \eqref{dynamic:general-rob-constraint} at $v$.
  
  Hence, for the \stpath ~$P''_z=P'_z\cup P_z$ we have
  \begin{align*}
    \thea(\xdgm)-\tau_{P''_z}-\Delta_z(P''_z) 
      = \thea(\xdgm)-\tau_{P_z}-\Delta_z(P_z)-\tau_{P'_z}-\Delta_z(P'_z)\geq 1
  \end{align*}
  with $\xdpm(P''_z,\thea(\xdgm)-\tau_{P''_z}-\Delta_z(P''_z))=c$.
  Thus, for each $z\in\Lambda$ we find an \stpath ~fulfilling the property 
  in the definition of $\thea(\xdgm)$ for $\xdpm$. 
  It follows that $\thea(\xdpm)\leq \thea(\xdgm)$.
\end{proof}

We now extend the complexity result of \DPM ~for fixed $\Gamma\geq 1$ to \DGM.
\begin{theorem}\label{thm:DGM:fixedgamma}
	\DGM ~is NP-hard for fixed $\Gamma\geq 1$.
\end{theorem}
\begin{proof}
  Let $\xoptdgm$ be an optimal solution to \DGM ~on the instance from 
  \Cref{thm:DPM:gamma1} for $\Gamma=1$.
  If $\foptdgm>0$, we know that $\theta_{\textit{ea}}(\xoptdgm)\leq T$.
  By \Cref{lemma:dynamic:earliestarrival}, it follows that there exists a 
  feasible solution $\xdpm$ to \DPM ~with 
  $\theta_{\textit{ea}}(\xdpm)\leq \theta_{\textit{ea}}(\xoptdgm)$.
  Thus, in an optimal solution to \DPM, flow reaches the sink in time and 
  $\foptdpm>0$.
  On the other hand, if $\foptdgm=0$, we directly obtain $\foptdpm=0$ since 
  $\xoptdpm$ is feasible for \DGM. 
  Hence, by solving \DGM ~on the instance from \Cref{thm:DPM:gamma1} for 
  $\Gamma=1$, we can solve the according partition problem.
  
  Analogously to \DPM, we can extend the instance for $\Gamma=1$ to any fixed 
  $\Gamma\geq 1$.
\end{proof}

The previous results show that \DPM ~and \DGM ~are NP-hard even for $\Gamma=1$ 
although their static counterparts are solvable in polynomial time for that special case.

If $\Gamma$ is part of the input, \Cref{thm:static:general-gammainput} can be extended to the dynamic setting since \SGM ~is a special case of \DGM.

\begin{theorem}\label{thm:DGM:gammainputNP}
	\DGM ~is strongly NP-hard if $\Gamma$ is part of the input.
\end{theorem}

\begin{remark}
  Note that the graph in the proof of \Cref{thm:DPM:gamma1} is a series-parallel 
  graph and has unit capacities on all arcs. 
  Thus, \DPM ~and \DGM ~are NP-hard for fixed $\Gamma= 1$ even when 
  restricted to instances with $u\equiv 1$. 
  This stands in contrast to the static setting where \SPM ~and \SGM ~are 
  solvable in polynomial time when restricted to instances with $u\equiv 1$, 
  even if $\Gamma$ is part of the input.
\end{remark}

\subsection*{Integral flows}
\label{sec:dynamic-complexity-integral}
While there are no results on the complexity of integral \DAM ~so far, 
\cite{gottschalk} prove that integral \DPM ~is NP-hard and inapproximable within 
any factor, even for $\Gamma=1$. 
For integral \DGM, we obtain NP-hardness by 
using the instance in the proof of \Cref{thm:DPM:gamma1}.
\begin{theorem}
	For fixed $\Gamma\geq 1$, integral \DGM ~is NP-hard.
\end{theorem}
\begin{proof}
	Recall that in the proof of \Cref{thm:DPM:gamma1} we have shown
	that the optimal value of \DPM ~on a constructed instance with $\Gamma=1$ 
	is greater than zero if and only if a given \textsc{partition} instance 
	is a 'YES'-instance. 
	On this constructed instance, there is an optimal solution to \DPM ~which 
	has only integral values.
	Thus, the optimal value of integral \DPM ~on the constructed instance 
	is greater than zero if and only if the \textsc{partition} instance is 
	a 'YES'-instance.
	 Applying the idea from \Cref{lemma:dynamic:earliestarrival}, it follows directly
	that also the optimal value of integral \DGM ~on the constructed instance 
	is greater than zero if and only if the \textsc{partition} instance is 
	a 'YES'-instance.
	The extension to $\Gamma>1$ works analogously as in the proof of 
	\Cref{thm:DPM:gamma1}.
\end{proof}


\bookmarksetup{depth=subsection}
\subsection{Comparison of the solution quality between the robust dynamic flow models} \label{sec:dynamic:bounds}

In the following, we briefly compare the dynamic models to each other.

We can extend the instances in the proofs
of \Cref{Prop:gapSPM_SAM} and \Cref{Prop:gapSAM_SPM} to the dynamic setting by defining
$\tau\equiv 0$ and $\Delta\tau\equiv T$ with $T\in\N$. With analogous proofs, we obtain the following corollaries.
 
\begin{corollary}\label{prop:gapDPM_DAM}
  For any $\Gamma,T\in\N$ and $\alpha\in\R$ there are instances such that 
	${\foptdpm > \alpha \foptdam}$ and consequently ${\foptdgm>\alpha \foptdam}$.
\end{corollary}

\begin{corollary}\label{prop:gapDAM_DPM}
	Let $\Gamma,T \in \N$. 
	Then, for any $\alpha = \Gamma+1-\nicefrac{1}{\beta}$ with 
	$\beta\in \N$ there are instances such that 
	${\foptdam=\alpha \foptdpm}$ and 
	${\foptdgm = \alpha \foptdpm}$. 
\end{corollary}
Hence, as in the static case, neither of the two models, \DAM ~or \DPM, can be considered preferably 
in general in terms of the robust optimal value. 


\subsection*{Price of robustness}\label{sec:dynamic:por}
\bookmarksetup{depth=part}

Next, we compare the nominal optimal flow value $\foptnom$ to the nominal value of 
the robust optimal flow solutions for the different dynamic models, i.e. the PoR. 
First, since on any instance it holds that $\foptnom\geq \foptdamnd$, 
from \Cref{prop:gapDPM_DAM} it follows:

\begin{corollary}\label{prop:gapNom_DAM}
	For any $\Gamma,T\in\N$ and $\alpha \in \R$, there are instances such that 
	$\foptnom >\alpha f(\xdam)$ for all robust feasible solutions $\xdam$ to \DAM.
\end{corollary}

Second, we establish a lower bound on the PoR. 
This bound is stronger than the lower bound in the static setting, cf. \Cref{Prop:gapNom_SPMNom}. 
\begin{prop}\label{prop:gapNom_DPM}
	Let $\Gamma \in \N$. For any $\alpha \in [1,\Gamma +1)\cap \Q$ there exists a 
	$T>\Gamma+1$ such that there are instances with  
	${\foptnom= \alpha \foptdpmnd= \alpha \foptdgmnd}$,
	where $\xoptdpm$ is an optimal solution to \DPM ~with maximal nominal value 
	and $\xoptdgm$ is an optimal solution to \DGM ~with maximal nominal value.
\end{prop}
\begin{proof}
	For given $\Gamma\geq 1$, we initially set $T= 1$ and construct an instance 
	$I=(G,\Gamma,T)$ as depicted in \Cref{fig:gapNom_DPM}. 
	\begin{figure}[h]
	\centering
    \resizebox{1\textwidth}{!}{
    \begin{tikzpicture}
        \node[circle,draw=black,thick,minimum size=8mm](s)at(-8.5,0){$s$};
        \node[circle,draw=black,thick,minimum size=8mm](v1)at(-5.5,0){$v_1$};
    	\node[circle,draw=black,thick,minimum size=8mm](v2)at(-2.5,0){$w_1$};
    	\node (P1) at (-0.2,0)[fill, circle, inner sep = 0.5pt] {};
    	\node (P2)at (0,0)[fill, circle, inner sep = 0.5pt] {};
    	\node (P3) at (0.2,0)  [fill, circle, inner sep = 0.5pt] {};
    	\node[circle,draw=black,thick,minimum size=8mm](v3)at(2.5,0){$v_{\Gamma}$};
    	\node[circle,draw=black,thick,minimum size=8mm](v4)at(5.5,0){$w_{\Gamma}$};
    	\node[circle,draw=black,thick,minimum size=8mm](t)at(8.5,0){t};
        \draw[thick,->](s)edge node[above]{\small{$(a_1^*,0,0,1)$}}(v1);
        \draw[thick,->,bend left=25](v1)edge node[above]
        	{\small{$\left(a_1^1,\frac{1}{\Gamma+1}-\eta,0,1\right)$}}(v2);
        \draw[thick,->,bend right=25](v1)edge node[below] 
        	{\small{$\left(a_1^2,0,\frac{1}{\Gamma+1},1\right)$}}(v2);
        \draw[thick,->](v2)edge node[above]{\small{$(a_2^*,0,0,1)$}}(-0.5,0);
        \path[thick,->](0.5,0)edge node[above]
        	{\small{$(a_{\Gamma}^*,0,0,1)$}}(v3);
        \draw[thick,->, bend left=25](v3)edge node[above] 
        {\small{$\left(a_{\Gamma}^1,\frac{1}{\Gamma+1}-\eta,0,1\right)$}}(v4);
        \draw[thick,->, bend right=25] (v3) edge node[below] 
        	{\small{$\left(a_{\Gamma}^2,0,\frac{1}{\Gamma+1},1\right)$}}(v4);
        \draw[thick,->](v4)edge node[above]{\small{$(a_{\Gamma+1}^*,0,0,1)$}}(t);
    \end{tikzpicture}}
	\vspace{-0.6cm}
    \caption{Instance with $\foptnom=1$ and 
    		$\foptdpmnd=\foptdgmnd=\nicefrac{1}{(\Gamma +1)}+\Gamma\eta$ 
    		in the proof of \Cref{prop:gapNom_DPM}, 
    		i.e. $\foptnom\geq \foptdpmnd$ and  $\foptnom\geq\foptdgmnd$.}
    \label{fig:gapNom_DPM}
	\end{figure}
	We define
	\begin{equation*}
	\eta:=\frac{(\Gamma+1)-\alpha}{\Gamma\alpha(\Gamma+1)}.
	\end{equation*}

  Note that $0 < \eta \leq \nicefrac{1}{(\Gamma+1)}$.
  The graph $G$ is composed of a source $s$, a sink $t$ and the nodes 
  $v_{i}$, $w_{i}$, $i\in[\Gamma]$. 
  There are two parallel arcs, $a_i^1$ and $a_i^2$, from $v_i$ to $w_i$ for all 
  $i\in[\Gamma]$, with travel times 
  $\tau_{a_i^1}=\nicefrac{1}{(\Gamma+1)}-\eta$ and $\tau_{a_i^2}=0$,
  and delays $\Delta\tau_{a_i^1}=0$ and 
  $\Delta\tau_{a_i^2}=\nicefrac{1}{(\Gamma+1)}$.
  For all $i\in[\Gamma-1]$, $w_i$ is connected to $v_{i+1}$ 
  by an arc $a_{i+1}^*$ with $\tau_{a_{i+1}^*}=\Delta\tau_{a_{i+1}^*}=0$. 
  In addition, each $w_i$ is connected to $v_{i+1}$ 
  by an arc $a_{i+1}^*, ~i\in[\Gamma-1]$. 
  Besides, there is an arc $a_{1}^*$ from $s$ to $v_1$ and 
  an arc $a_{\Gamma+1}^*$ from $w_{\Gamma}$ to $t$. 
  All of these arcs $a_i^*,~i\in [\Gamma+1]$, have travel time and delay 
  equal to zero.
  All arcs have unit capacity.

  We define 
  $P_1\coloneqq\{a_1^*,a_1^1, \ldots, a_\Gamma^*, a_\Gamma^1, a_{\Gamma+1}^*\}$ 
  and  
  $P_2\coloneqq\{a_1^*,a_1^2, \ldots, a_\Gamma^*, a_\Gamma^2, a_{\Gamma+1}^*\}$.    
   
  As the nominal travel time of the arcs $a_i^2,~i\in [\Gamma]$ is zero, 
  the nominal optimal solution sends one unit of flow on $P_2$ for all 
  $\theta\in\T$ resulting in an optimal value $\foptnom=T=1$.
     
  Moreover, we note that only the arcs $a_i^2,~i\in [\Gamma]$, have nonzero delay. 
  Hence, in the worst case, all of these $\Gamma$ arcs are delayed, 
  resulting in a travel time of 
  $\tau_{a_i^2}+\Delta\tau_{a_i^2}=\nicefrac{1}{(\Gamma+1)} > 
  \nicefrac{1}{(\Gamma+1)}-\eta=\tau_{a_i^1}+\Delta\tau_{a_i^1}, i\in[\Gamma]$.
  Thus, the optimal solution $\xoptdpm$ to \DPM ~sends one unit of flow 
  on path $P_1$ in each time step $\theta\in\T$ yielding an objective value of
  ${\foptdpm=1-\left(\nicefrac{\Gamma}{(\Gamma+1)}-\Gamma\eta\right)=
  \nicefrac{1}{(\Gamma+1)}+\Gamma\eta}$.
  As $\xoptdpm$ only uses arcs with delay equal to zero, $\foptdpmnd=\foptdpm$. 
  This robust optimal solution can be shown to be unique, so that the described solution has 
  maximal nominal value.
  Thus, combining the above proves the first assertion as
  \begin{align*}
    \frac{\foptnom}{\foptdpmnd}
  	= \frac{1}{\frac{1}{(\Gamma+1)}+\Gamma\eta}
  	= \frac{1}{\frac{1}{(\Gamma+1)}+\Gamma 
  			\left( \frac{(\Gamma+1)-\alpha}{\Gamma\alpha(\Gamma+1)}\right)}
  	= \frac{1}{\frac{1}{(\Gamma+1)} + \frac{1}{\alpha} 
  			- \frac{1}{(\Gamma+1)}}
  	=\alpha.
  \end{align*}	
  The unique optimal solution $\xoptdpm$ to \DPM ~is also the unique optimal 
  solution to \DGM ~yielding the second assertion.
	
	We further note that the constructed instance contains arcs with 
	non-integral travel times or delays. 
	However, since $\alpha$ and therefore $\eta$ are rational, 
	we can w.l.o.g. scale the time horizon $T$ and the instance to an 
	instance with integral travel times 
	as well as integral delays giving the same bound.
\end{proof}


\bookmarksetup{depth=subsection}
\subsection{Approximation approaches}\label{sec:dyn:approx}
In the following, we present two solution concepts for \DPM ~and \DGM.
\subsection*{Temporally repeated flows}\bookmarksetup{depth=part}
As mentioned in \Cref{sec:dynamic:models}, temporally repeated flows are a 
classical solution concept for the nominal dynamic maximum flow problem. 
We recall the results from \cite{gottschalk} which show that temporally repeated 
flow solutions do not solve \DPM ~to optimality in general and give bounds on the
corresponding approximation factor. In the following, we denote the optimal value of a temporally repeated flow by $\foptdpmtr$. 

\begin{theorem}\cite[Proposition 7, Theorem 3]{gottschalk}
	Let $\Gamma \in \N$ be arbitrary. We set $T\coloneqq\Gamma+1$. 
	Then, for any $\alpha \in [1,T]$ there are 
	instances such that $\foptdpm \geq \alpha \foptdpmtr$. 
	On the other hand, for all instances the inequality 
	$\foptdpm \leq \eta k \log(T) \foptdpmtr$ holds, where $k$ and $\eta$ 
	are parameters depending on the given graph.
\end{theorem}

It is clear that the lower bound in this proposition translates to \DGM. 
In detail, we observe that the instances used to obtain \Cref{prop:gapDPM_DAM} and \Cref{prop:gapDAM_DPM}, both have optimal solutions to \DPM ~that are temporally repeated. This leads to the following result.  
\begin{corollary}\label{Cor:gapDGM_DPMrep-DAM_DPMrep-DPMrep_DAM}
	Let $\Gamma\in\N$.
	For any $\alpha \in[1,\Gamma+1)$ there are instances such that 
	${\foptdgm \geq \alpha \foptdpmtr}$ and ${\foptdam \geq \alpha \foptdpmtr}$.
	Moreover, for any $\alpha \in \R$ there are instances such that 
	${\foptdpmtr \geq \alpha \foptdam}$.
\end{corollary}
Hence, it does not hold in general that ${\foptdam > \foptdpmtr}$ or ${\foptdpmtr > \foptdam}$.


\subsection*{Temporally increasing flows}
The previous paragraph shows that the simple concept of temporally repeated 
flows satisfies approximation guarantees for \DPM. 
However, we are interested in the question whether there is a similar approach 
for \DGM ~also fulfilling approximation guarantees or even being optimal.
We first state a natural extension of temporally repeated flows, which we call 
temporally increasing flows. 
In detail, they fulfill the condition
\begin{align*}
  x(P,\theta)\leq x(P,\theta+1)\quad 
    \forall P \in \barP,\theta \in [T-1].
\end{align*}
We denote an optimal temporally increasing flow 
by~$\xoptti$ and the corresponding objective value by~$\foptti$.
Clearly, all temporally repeated flows are also temporally increasing flows.

The intuitive idea for temporally increasing flows is
that we choose a (sub)path decomposition and send as much flow as possible on 
the chosen paths satisfying the capacity and robustness constraints.
However, an optimal temporally increasing flow is not 
necessarily an optimal solution to \DGM, 
i.e. we sometimes have to reroute flow within the time horizon to solve \DGM ~to optimality.
The following example demonstrates such a case.

\begin{example}\label{ex:gapDGM_DGMti}
  	Let $\Gamma=1$ and $T=2$. We construct a graph $G$ consisting of 
  	three nodes $\{s,v,t\}$ on which we have 
  	$\foptdgm=\foptdpm=\foptdam=2$ and $\foptti=1.5$, 
	cf. \Cref{fig:gapDGM_DGMti}.
	In $G$, there is one arc $a_1$ from $s$ to $v$ with 
	$\tau_{a_1}=\Delta\tau_{a_1}=0$. 
	$v$ is connected to $t$ by one arc $a_2$ with 
	$\tau_{a_2}=0,~\Delta\tau_{a_2}=2$ 
	and by one arc $a_3$ with $\tau_{a_3}=1,~\Delta\tau_{a_3}=0$. 
	Besides, there is one arc $a_4$ from $s$ to $t$ with 
	$\tau_{a_4}=\Delta\tau_{a_4}=1$. All arcs have unit capacity.
	\begin{figure}[h]
	\centering
    \begin{tikzpicture}[scale=.8]
        \node(s)at(-6.5,0) [circle,draw=black, thick, minimum size=8mm] {$s$};
    	\node(v)at(-3,0) [circle,draw=black, thick, minimum size=8mm] {$v$};
    	\node(t)at(1,0) [circle,draw=black, thick, minimum size=8mm] {$t$};
        \draw[thick,->] (s) edge node[above] {$(a_1,0,0,1)$} (v);
        \draw[thick,->,bend left=30](v)edge node[above]{$(a_2,0,2,1)$}(t);
        \draw[thick,->,bend right=0](v)edge node[below]{$(a_3,1,0,1)$}(t);
        \draw[thick,->,bend right=35](s) edge node[above]{$(a_4,1,1,1)$}(t);
    \end{tikzpicture}
    \caption{Instance with $\foptdgm\hs{-0.2em}=\hs{-0.1em}\foptdam\hs{-0.2em}=
    		\hs{-0.1em}\foptdpm\hs{-0.2em}=\hs{-0.1em}2$ and 
    		$\foptti\hs{-0.1em}=\hs{-0.1em}1.5$ in \Cref{ex:gapDGM_DGMti}.}
    \label{fig:gapDGM_DGMti}
	\end{figure}

	Obviously, $\xoptdgm(a_1,\theta)=1$ 
    and $\xoptdgm(a_4,\theta)=1$ for all $\theta\in\T$.
    For $\theta=1$, the flow unit arriving at $v$ is sent 
    via $a_3$ to $t$ and arrives in any case within the time 
    horizon $T=2$ as $\Delta\tau_{a_3}=0$. 
    For $\theta=2$, the flow unit arriving at $v$ is sent to $t$ via $a_2$.
    In the worst case, $a_2$ or $a_4$ is delayed so that no flow arrives at 
    the sink via that arc. 
    Regardless of which of these two arcs is delayed, one unit of flow arrives 
    at $t$ within the time horizon $T$ via the remaining arc, which results in 
    an objective value of $2$.
       
    As in the general solution, $\xoptti(a_1,\theta)=\xoptti(a_4,\theta)=1$ 
    for all $\theta\in\T$.
    In each time step, the optimal temporally increasing flow solution sends 
    $0.5$ flow units via each of the arcs $a_2$ and $a_3$.
    In any case, $0.5$ flow units arrive at the sink via $a_3$ as 
    $\Delta\tau_{a_3}=0$. 
    Regardless of which of the other arcs, $a_2$ and $a_4$, is delayed, 
    one unit of flow arrives at $t$ within the time horizon via the remaining, 
    non-delayed arc. 
    This results in an optimal value of $\foptti=1.5$.

	We note that the optimal solution $\xoptdgm$ to \DGM ~is also feasible 
	to \DAM, resulting in the same objective value. 
	Combining $a_1$ and the $v$-$t$-arcs to \stpaths ~while sending 
	the same amount of flow as in $\xoptdgm$, the solution is feasible for 
	\DPM ~and yields the same objective value. 	
	
	We remark that the example can be extended to arbitrary $\Gamma\geq 1$ 
	and $T\geq 2$. 
\end{example}

In contrast to the previous example, there are also instances on which 
the best temporally increasing flow has a greater objective value than 
the optimal solutions to \DAM ~and \DPM, cf. the instances in \Cref{sec:dynamic:bounds}.
It remains an open question how to compute an optimal temporally increasing flow. 
In contrast to temporally repeated flows, for which a linear programming 
formulation exists if the paths in the graph are T-bounded, there is no 
trivial optimization model for temporally increasing flows. 
Also, it is unclear whether there is an upper bound on the approximation ratio,
a question which required a quite tedious proof for temporally repeated 
flows for \DPM, cf. \cite{gottschalk}.  
Finding an efficient algorithm for temporally increasing flows or other 
approximation approaches as well as proving approximation guarantees
could be an important focus for future research.

%% file: conclusion.tex
\bookmarksetup{depth=subsection}
\section{Conclusion and outlook}\label{sec:conclusion}
In this paper, we have introduced new models for the robust maximum flow problem 
as well as for the robust maximum flow over time problem. 
The main advantage of the proposed general models, \SGM ~and \DGM, is
to unify the known robust flow models in order to obtain less
conservative solutions. 
We provided a thorough analysis of the complexity of these new models 
and investigated the solution quality in comparison to 
existing models, showing e.g. that the new general models yield better 
robust optimal values than the known models. 
In particular, we also showed that \SGM ~remains solvable in polynomial time in 
the special cases for which \SPM ~is solvable in polynomial time, 
for example if only one arc may fail.
These results highlight the advantages of the new general models compared to 
the previously known models.

We have pointed out several open questions, and highlight some of them here.
We know that \SPM ~is NP-hard if the number of failing arcs $\Gamma$ is part 
of the input.
However, it is unknown whether the problem is already NP-hard for ${\Gamma=2}$, 
cf. \cite{disser}. 
Thus, the complexity for fixed $\Gamma$ is still an open question, for \SPM ~as well as 
for \SGM. 
Furthermore, the complexity of the integral versions of the arc models is unclear whereas we clarified the complexity of the other integral models.

We have established several lower bounds on possible gaps and the price of
robustness. 
It stands out that most of these lower bounds are ${\Gamma+1}$.
An open question is whether there exist lower bounds in the
dynamic setting that depend on the time horizon~$T$.
Finally, we have proven an upper bound of two on the gap between 
\SGM ~and \SPM, or \SAM ~and \SPM ~respectively, 
for the case of one failing arc, i.e. ${\Gamma=1}$.
Since the general models, \SGM ~and \DGM, are computationally hard to solve
in general, a further investigation of those gaps and a generalization
of our result to cases with ${\Gamma>1}$ could lead to approximation 
guarantees for the general models.

%% file: acknowledgements.tex
\section*{Acknowledgments}
\label{sec:acknowledgements}
The authors thank the DFG for their support within Projects B06 and Z01 in CRC TRR 154.
This research has partially been performed as part of the Energie Campus
Nürnberg (EnCN) and is supported by funding of the Bavarian State
Government.
\subsection*{Declarations}
The authors have no competing interests to declare.

%% file: appendix.tex
\appendix
\bookmarksetup{depth=subsection}
\section{Proofs of \texorpdfstring{\Cref{sec:static}}{Section \ref{sec:static}}}

\subsection{Proof of \texorpdfstring{\Cref{lemma:static:gamma1-maxflow}}{Theorem \ref{lemma:static:gamma1-maxflow}}}\label{prooflemma:static:gamma1-maxflow}
\begin{proof}
  Let $x$ be an optimal solution to \SGM ~with maximal nominal flow value 
  $\fnom(x)$ and assume for contradiction that $x$ is not nominal optimal, i.e. 
  $\fnom(x)<\foptnom$. 
  We split the flow into $x=g+h$, where $g$ denotes the flow on paths ending 
  at $t$, i.e. $g(P)=0$ for all $P\notin\dinPt$ 
  and $h$ denotes the flow on all other subpaths, i.e. $h(P)=0$ for all 
  $P\in\dinPt$.
  In order to illustrate the subsequent steps of the proof, we provide a sketch in 
  \Cref{fig:thm_nominal}. 
  
  Since, by assumption, $x$ is not nominal  optimal, there exists an 
  augmenting acyclic $s$-$t$-path $P$
  in the residual graph of $g$, which we denote by $G_g$. 
  If this augmenting path existed in $G_x$, it would be possible to increase the 
  nominal flow value of $x$ without reducing the robust flow value. 
  This would contradict the assumption that $x$ maximizes $\fnom(x)$. 
  
  Therefore, there is at least one forward-arc of $P$ in $G_g$ that does not  
  exist in $G_x$. 
  Due to $x=g+h$, this arc thus carries flow that is part of $h$. 
  Let $\bar{a}$ be the last arc on $P$ such that there exists a path $P'\in\barP$ 
  with $h(P')>0$ and $\bar{a}\in P'$. Let $P'$ end at node $w$, 
  i.e. $P'\in\dinP(w)$.
  Furthermore, let $P_{\leq \bar{a}}$ and $P_{>\bar{a}}$ denote the subpaths of 
  $P$ that end in the end node of $\bar{a}$ and 
  start at the end node of $\bar{a}$, respectively, i.e. $P=P_{\leq \bar{a}}\cup P_{>\bar{a}}$ and 
  $\bar{a}\in P_{\leq \bar{a}}$, $\bar{a}\notin P_{>\bar{a}}$.
  Analogously, we use this notation for $P'$.
  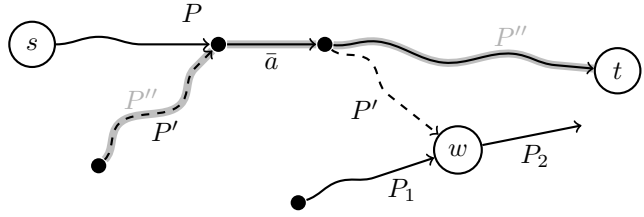
\begin{figure}[h]
  	\begin{tikzpicture}[scale=.7]
  	\node(s) at (-6,3.5) [circle,draw=black,thick,minimum size=6mm] {$s$};
  	\node(va) at (-2.5,3.5)
  		[circle,inner sep=0pt,fill=black,thick,minimum size=2mm] {};
  	\node(wa) at (-0.5,3.5) 
  		[circle,inner sep = 0pt,fill = black,thick,minimum size=2mm] {};
  	\node(v) at (-4.75,1.2) 
  		[circle,inner sep = 0pt,fill = black,thick,minimum size=2mm] {};
  	\node(w) at (2,1.5) [circle,draw=black, thick, minimum size=6mm] {$w$};
  	\node(t) at (5,3) [circle,draw=black, thick, minimum size=6mm] {$t$};
  	\node(v1) at (-1,0.5) 	
  		[circle,inner sep = 0pt,fill = black,thick,minimum size=2mm] {};
  	\node(w2) at (4.5,2){};
   	\node(w3) at (4.5,-0.25){};
  	\draw[line width=3pt,black!25,decorate, decoration={snake,amplitude=2mm,
  		segment length=15mm,post length=3mm}](v)--(va);
  	\draw[line width=3pt,black!25](va)edge node[below]{}(wa);
	\draw[line width=3pt, black!25, decorate, decoration={snake,amplitude=1mm,
		segment length=18mm,post length=3mm}](wa)--(t);  	
  	\draw[thick,->,bend left =0](va)edge node[below]{$\bar{a}$}(wa);
  	\draw[thick, dashed,decorate, decoration={snake,amplitude=2mm,
  		segment length=15mm,post length=3mm},->](v)--(va);
	\draw[thick, dashed,decorate, decoration={snake,amplitude=2mm,
		segment length=15mm,post length=3mm},->](wa)--(w);
	\draw[thick, decorate, decoration={snake,amplitude=1mm,
		segment length=18mm,post length=3mm},->](s)--(va);
	\draw[thick, decorate, decoration={snake,amplitude=1mm,
		segment length=18mm,post length=3mm},->](wa)--(t);
	\draw[thick, decorate, decoration={snake,amplitude=1mm,	
		segment length=15mm,post length=3mm},->](v1)--(w);
	\draw[thick, decorate, decoration={snake,amplitude=1mm,
		segment length=18mm,post length=3mm},->](w)--(w2);
  	\node[above, black!30] at (3,3.3){$P''$};
  	\node[above, ] at (0.25,1.9){$P'$};
  	\node[above, ] at (0.95,0.35){$P_1$};
  	\node[above, ] at (3.45,1){$P_2$};
   	\node[above, ] at (-3,3.75){$P$};
   	\node[above, ] at (-3.5,1.5){$P'$};
   	\node[above, black!30] at (-3.95,2.1){$P''$};
  	\end{tikzpicture}
  	\vspace{-0.5cm} 
  	\caption{Sketch for the proof of \Cref{lemma:static:gamma1-maxflow}. }
  	\label{fig:thm_nominal}
  \end{figure}
  
  We will now construct a new solution $x'=g'+h'$ that increases the 
  nominal value without reducing the robust value compared to $x$, 
  which then contradicts our assumption. 
  The main idea of the following construction is to reduce the flow on $P'$ 
  by some positive amount $\epsilon$ while increasing flow along $P_{>\bar{a}}$. 
  
  First, we increase flow along $P_{>\bar{a}}$: the set of arcs that fail in the 
  worst cases, which can be written as
  $\{a\in A\defsep a\in \argmax \sum_{P\in\dinPt:a\in P}g(P)\}$, 
  is independent of the specific path decomposition of $g$, as $\Gamma=1$.
  Therefore, we can consider the underlying arc flow of $g$ and execute an 
  augmentation step on this flow. 
  In particular, we augment the underlying arc flow of $g$ by a small value $\epsilon>0$ 
  on $P''=P'_{\leq \bar{a}}\cup P_{>\bar{a}}$. 
  This increase is possible since $h(P')>0$, implying that the arcs of 
  $P'_{\leq\bar{a}}$ are in $G_g$, and since $P$ was assumed to be in $G_g$.
  As argued above, we can now choose an arbitrary $\dinPt$-path decomposition 
  of this new flow and obtain an interim $\tilde{g}$. 
  The above modification, namely to replace $g$ by $\tilde{g}$, strictly increases 
  the nominal flow value by $\varepsilon$ without decreasing the robust flow value.
  
  Second, we reduce the flow on $P'$: we reduce the flow $h(P')$ 
  by $\varepsilon$, i.e. we set $\tilde{h}(P')= h(P)-\varepsilon$. 
  We distinguish the two cases that the current new flow, $\tilde{x}=\tilde{g}+\tilde{h}$, 
  is feasible and that it is not: 
  i) If $\tilde{x}=\tilde{g}+\tilde{h}$ is feasible for \SGM, we set $x' = \tilde{x}$ and 
  the construction is finished.
  
  ii) If $\tilde{x}=\tilde{g}+\tilde{h}$ is not feasible for \SGM, we construct a new 
  feasible solution $x'$ as follows.  
  The reduction of $h(P')$ to $\tilde{h}(P')$ can only lead to infeasibility in 
  terms of robust flow conservation \eqref{static:general-rob-constraint} 
  at the node $w$.
  Let $P_2$ be a path starting at $w$ with $x(P_2)>0$.
  Such a path exists due to robust infeasibility of $\tilde{x}$ at $w$. 
  Since $x$ is robust feasible, there is a path $P_1$, which is arc-disjoint 
  to $P'$ and ends at $w$ with $\tilde{h}(P_1) = h(P_1)>0$. 
  In order to ensure robust feasibility at $w$ and all other nodes, 
  we use the concatenation of the paths $P_1$ and $P_2$ and send a small amount 
  of flow on this concatenation instead of the individual paths.  
  In detail, 
  \begin{equation}
  x'(P_1)=\tilde{h}(P_1)-\varepsilon, \quad
  x'(P_2)=\tilde{x}(P_2)-\epsilon, \quad x'(P_1\cup P_2)=\tilde{x}(P_1\cup P_2)+\epsilon.
  \end{equation}
  The resulting flow $x'$ is decomposed, as before, into flow on $\dinPt$-paths, 
  called $g'$, and the remaining flow, called $h'$. 
  If $P_2\in\dinPt$, we thus have $g'(P_1\cup P_2)=x'(P_1\cup P_2)$ 
  and hence $g'$ compared to $g$ has been increased by $\varepsilon$ on all 
  arcs of $P_1$.
  In the first construction step above, $g$ has been increased by $\epsilon$ on 
  $P''$ in order to obtain $\tilde{g}$. 
  By construction, $P_1$ is arc-disjoint to $P'$ and thus to $P''_{\leq\bar{a}}$.
  Further, since $P''_{>\bar{a}}=P_{>\bar{a}}$ does not contain any arc $a$ with 
  $h(a)>0$, $P_1$ is arc-disjoint to $P''$. 
  Thus, $g'$ has not been increased by more than $\epsilon$ on any arc 
  compared to $g$. 
  
  In both cases, i) and ii), we have constructed a new feasible solution $x'$ with
  \begin{align*}
    \sum_{\substack{P\in\dinPt:\\a\in P}}g'(P)
      -\sum_{\substack{P\in\dinPt:\\a\in P}}g(P)\leq\epsilon
      \qquad\forall a\in A,
  \end{align*}
  
  and $\fnom(x')=\fnom(x)+\varepsilon$. 
  Thus,
  \begin{align*}
    f_R(x')=\fnom(x')-\max_{a\in A}\sum_{\substack{P\in\dinPt:\\a\in P}}g'(P)
    &=\fnom(x)+\epsilon-\max_{a\in A}\sum_{\substack{P\in\dinPt:\\a\in P}}g'(P)\\
    &\geq \fnom(x)+\varepsilon
      -(\varepsilon+\max_{a\in A}\sum_{\substack{P\in\dinPt:\\a\in P}}g(P))
    ~=f_R(x).
  \end{align*}  
  Hence, $x'$ has a strictly larger nominal flow value than $x$ and at least 
  the same robust flow value as $x$, contradicting the initial assumption on $x$.
\end{proof}

\subsection{Proof of \texorpdfstring{\Cref{lem:SGM_capacity1infty_NP}}{Proposition \ref{lem:SGM_capacity1infty_NP}}}\label{prooflem:SGM_capacity1infty_NP}
\begin{proof}
	We show the statement by a reduction from the 
	static robust flow problem with arbitrary capacities.
	Thereby, we construct from an instance with arbitrary capacities an 
	instance with capacities in $\{1,\infty\}$. 
	This construction has been proposed in \cite[Lemma 5, proof]{disser} 
	for the path model and it remains to be proven that equivalence also 
	holds for the general model. 
	Let $\Gamma\geq 1$ and $G=(V,A,s,t,u)$ be an arbitrary instance for \SGM. 
	We construct a graph $G'=(V',A',s,t,u')$ from $G$ by adding 
	a node $v_a$ for every $a=(v,w)\in A$ and replacing $a$ by 
	an arc $a'=(v,v_a)$ with infinite capacity 
	and $u_a$ parallel arcs $a_i''=(v_a,w), i \in [u_a]$ 
	with unit capacity (cf. \Cref{fig:GraphTransformation_u=(1_infty)}). 
	The arc set $A'$ thus consists of one arc $a'$ and $u_a$ arcs $a''_i$ 
	for each arc $a\in A$. 
	We denote the set of scenarios and the set of all subpaths on 
	$G'$ by $\Sc'$ and $\barP'$, respectively.
	Given a \vwpath ~$P$ in $G$ with $v,w\in V$, 
	we denote the set of all corresponding \vwpaths ~in $G'$ by $\barP'(P)$,
	meaning that every path $P'\in \barP'(P)$ uses for 
	every $a\in P$ the corresponding arc $a'$ and one of the arcs $a''_i$, $i\in[u_a]$.
	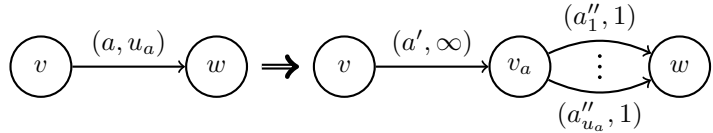
\begin{figure}[h]
		\centering
		\begin{tikzpicture}[scale=.7]
		\node(v) at(-4.5,0)[circle,draw=black,thick,minimum size=8mm] {$v$};
		\node(w) at(-1.2,0)[circle,draw=black,thick,minimum size=8mm] {$w$};
		\node(vl) at(1.2,0)[circle,draw=black,thick,minimum size=8mm] {$v$};
		\node(va) at(4.5,0)[circle,draw=black,thick,minimum size=8mm] {$v_a$};
		\node(wl) at(7.5,0)[circle,draw=black, thick, minimum size=8mm] {$w$};
		\draw[thick,double,distance=7pt,->] (-0.35,0) -- (0.35,0);
		\draw[thick,->](v)edge node[above]{$(a,u_a)$}(w);
		\draw[thick,->](vl)edge node[above]{$(a',\infty)$}(va);
		\draw[thick,->,bend left=25](va)edge node[above]{$(a''_1,1)$}(wl);
		\draw[thick,->,bend right=25](va)edge node[below]
		{$(a''_{u_a},1)$}(wl);
		\node(P1) at (6,0.2)[fill, circle, inner sep = 0.5pt] {};
		\node(P2) at (6,0)[fill,circle,inner sep=0.5pt]{};
		\node(P3) at (6,-0.2)[fill, circle, inner sep = 0.5pt] {};
		\end{tikzpicture}
		\caption{Construction in the proof of \Cref{lem:SGM_capacity1infty_NP}.
			This construction has originally been proposed in \cite[Lemma 5, proof]{disser}.}
		\label{fig:GraphTransformation_u=(1_infty)}
	\end{figure}	
	
	Let $\xgm'$ be an optimal solution to \SGM ~on $G'$.
	Additionally, we define a vector $\xoptgm\in\R^{|\barP|}_{\geq 0}$ by
	\begin{align*}
    \xoptgm(P)\coloneqq \sum_{P'\in \barP'(P)}\xgm'(P') 
    	\qquad \forall P\in\barP.
	\end{align*}
	
	In the following, we show that $\xoptgm$ is an optimal solution to 
	\SGM ~on $G$. 
	First, from \Cref{cor:static:aux-lemma} it follows that $\xgm'(P')=0$ 
	for all paths starting at the intermediate nodes $v_a$, 
	i.e. for all $P'\in\delta_{\barP'}^{+}(v_a)$ with $a\in A$. 
	Additionally, we observe that for any $v\in V\setminus \{s\}$ there 
	is at least one scenario 
	\begin{align*}
	  S^*\in\argmax_{S'\in\Sc'}
	    \sum_{\substack{P'\in \delta_{\barP'}^{-}(v):\\P'\cap S'\neq \emptyset}}
	    \xgm'(P')
	\end{align*}
	such that $S^*$ only contains arcs of the type $a'$ and not arcs of the type 
	$a''_i$. 
	Hence, we can write the loss that is caused by the realization of the worst 
	case as follows: 
	\begin{align*}
	  \max_{S'\in\Sc'}
	    \sum_{\substack{P'\in\delta_{\barP'}^{-}(v):\\P'\cap S'\neq\emptyset}}
	    \xgm'(P')
	  =\max_{S\in\Sc}
	    \sum_{\substack{P\in \dinP(v):\\P\cap S\neq \emptyset}}
	    \sum_{P'\in \barP'(P)} \xgm'(P').
	\end{align*}
	Combining the above, it follows that for any $v\in V\setminus \{s\}$ we have 
	\begin{align*}
	  &\sum_{P'\in \delta_{\barP'}^{-}(v)} \xgm'(P') 
	    -\max_{S'\in\Sc'}
	      \sum_{\substack{P'\in \delta_{\barP'}^{-}(v):\\P'\cap S'\neq \emptyset}} 
	      \xgm'(P')\\
	  = &\sum_{P\in \dinP(v)}\sum_{P'\in \barP'(P)} \xgm'(P')
	    - \max_{S\in\Sc}
	      \sum_{\substack{P\in\dinP(v):\\P\cap S\neq \emptyset}}
	      \sum_{P'\in \barP'(P)} \xgm'(P')\\
	  = &\sum_{P\in \dinP(v)} \xoptgm(P) 
	    -\max_{S\in\Sc}
	      \sum_{\substack{P\in \dinP(v):\\P\cap S\neq \emptyset}} 
	      \xoptgm(P).
	\end{align*}
	Thus, $\xoptgm$ is a feasible solution to \SGM ~on $G$ with same optimal 
	value as $\xgm'$ on $G'$.
	Furthermore, from any feasible solution $x$ to \SGM ~on $G$ we can 
	trivially construct a feasible solution $x'$ to \SGM ~on $G'$ with at least 
	the same robust value by uniformly splitting the flow $x(P)$ into the paths 
	that correspond to $P$. 
	It follows that $\xoptgm$ is optimal for \SGM ~on $G$.
	Thus, it is at least as hard to solve \SGM ~on a graph with capacities 
	constrained to $u_a\in\{1,\infty\}$ as on a graph with arbitrary capacities, 
	which is NP-hard.
	Note that we can substitute the arcs in $G'$ with capacity $\infty$ 
	by arcs with capacity ${u_{max}=\max\{u_a:a\in A\}}$.
\end{proof}

\subsection{Proof of \texorpdfstring{\Cref{lem:SGM_integralSolution_NP}}{Theorem \ref{lem:SGM_integralSolution_NP}}}\label{prooflem:SGM_integralSolution_NP}
\begin{proof}
  	The proof is based on \cite[Theorem 8, proof]{disser}. 
	They set $\Gamma=2$ and construct a graph $G$ for integral \SPM ~from a given 
	\textsc{arc-disjoint paths} instance $G'$ 
	and show that an optimal solution to integral \SPM ~has an objective 
	value $\foptpm$ of at least $3$ if and only if there exist two arc-disjoint 
	paths in $G'$.
	Hence, it is NP-hard to distinguish instances of integral \SPM ~with optimal 
	value at least $3$ from those with optimal value at most $2$, which implies 
	hardness of the corresponding approximation for the problem. 
  
  \begin{figure}[h]
	\begin{tikzpicture}[scale=.7]
	\node(s) at (-5.25,-1.5) [circle,draw=black,thick,minimum size=8mm] {$s$};
	\node(v) at (-3.5,-0.5) [circle,draw=black,thick,minimum size=8mm] {$v$};
	\node(s1) at (-2,-1.5) [circle,draw=black,thick,minimum size=8mm] {$s_1$};
	\node(s2) at (-2,-3) [circle,draw=black,thick,minimum size=8mm] {$s_2$};
	\node(v2prime)at(-1.5,0.5)[circle,draw=black,thick,minimum size=8mm]{$v''$};
	\node(vprime)at(-1.5,1.75)[circle,draw=black,thick,minimum size=8mm]{$v'$};
	\node(t1) at (2.5,-3) [circle,draw=black, thick, minimum size=8mm] {$t_1$};
	\node(t2) at (2.5,-1.5) [circle,draw=black, thick, minimum size=8mm] {$t_2$};
	\node(w) at (4.25,-3.5) [circle,draw=black, thick, minimum size=8mm] {$w$};
	\node(t) at (5.25,-2) [circle,draw=black, thick, minimum size=8mm] {$t$};
	\draw[thick,->,bend left =0](s)edge node[below right]{$3$}(v);
	\draw[thick,->,bend left =30](vprime)edge node[above]{$2$}(t);
	\draw[thick,->,bend left =30](v2prime)edge node[above]{$2$}(t);
	\draw[thick,->,bend right=0](w) edge node[below right]{$2$}(t);
	\draw[thick,->,bend right=0](t1)edge node[above ]{$1$}(w);
	\draw[thick,->,bend right=40](s) edge node[above]{$1$}(w);
	\draw[thick,->,bend right=0](t2)edge node[above]{$1$}(t);
	\draw[thick,->,bend right=0](v)edge node[above]{$1$}(vprime);
	\draw[thick,->,bend right=0](v)edge node[below]{$1$}(v2prime);
	\draw[thick,->,bend left=35](s)edge node[above]{$1$}(v2prime);
	\draw[thick,->,bend left=35](s)edge node[above]{$1$}(vprime);
	\draw[thick,->,bend right=0](v)edge node[below left]{$1$}(s1);
	\draw[thick,->,bend right=0](s)edge node[below]{$1$}(s2);
	\draw[thick, dashed,decorate, decoration=
		{snake,amplitude=2mm,segment length=15mm,post length=3mm},->](s1)--(t1);
	\draw[thick,dashed,decorate, decoration=
		{snake,amplitude=2mm,segment length=18mm,post length=3mm},->](s2)--(t2);
	\draw[thick, dashed] (-2.68,-3.65) rectangle ++(5.83,2.80);
	\node[above] at (0.25, -1.75){$G'$};
	\end{tikzpicture}
	\vspace{-0.5cm} 
	\caption{Instance in the proof of \Cref{lem:SGM_integralSolution_NP}. 
		The arc labels denote the arc capacities. 
		This construction has originally been proposed in \cite{disser}.}
		\label{fig:integralDisser}
	\end{figure}
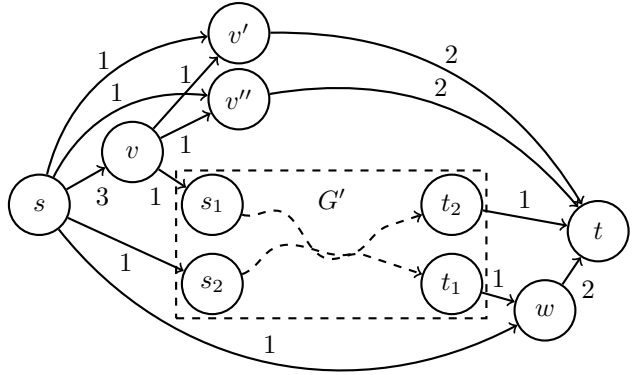
	
	Given an \textsc{arc-disjoint paths} instance $G'=(V',A')$, 
	we construct an instance $I=(G,\Gamma=2)$ for integral \SGM 
	~analogously to \cite[Theorem 8, proof]{disser}. 
	The resulting graph, cf. \Cref{fig:integralDisser}, consists of two parts:
	the inner part $G'$ with given start and end nodes $s_1, s_2, t_1, t_2$
	and the subgraph which we denote by $H$ induced by the added node set 
	$V_H=\{s,v,v',v'',w,t\}$. Note that for all nodes 
	$v \in V_H\setminus\{t\}$, we have $\dinA(v)\leq2=\Gamma$. 
	We now show that any optimal solution $\xoptpm$ to integral \SPM ~on the 
	constructed instance also solves integral \SGM ~to optimality. 
	Let $\xoptgm$ be an optimal solution to integral \SGM ~on the constructed 
	instance. 
	As $\dinA(v)\leq 2$ for all $v\in V_H\setminus\{t\}$, 
	it follows from \Cref{cor:static:aux-lemma} that all flow on subpaths in 
	$H$ can be deleted resulting in another optimal 
	solution to integral \SGM ~with flow only on entire \stpaths ~in~$H$. 
	
	Besides the subgraph $H$, we consider $G'$ and the subgraph 
	which we denote by $G'_{st}$ induced by the node set $V'_{st}\coloneqq V'\cup\{s,v,w,t\}$. 
	At most one unit of flow arrives at each of the nodes $s_1$ and $s_2$ in $G'$ 
	due to the capacities of the incoming arcs $(v,s_1)$ and $(s,s_2)$. 
	Such a flow unit cannot be distributed among different paths, 
	as an integral solution is required.
	This implies that there are for each node $v$ in $G'_{st}$ at most two 
	incoming paths carrying flow. 
	As $\Gamma=2$, it follows directly from \Cref{cor:static:aux-lemma} 
	that all flow on $s$-$v$-paths with $v\in V'_{st}\setminus\{t\}$ can be 
	deleted resulting in another optimal flow solution with flow only on 
	\stpaths.
	Thus, there is an optimal solution to integral \SGM ~which sends 
	flow only on \stpaths ~in $G$ and is hence feasible for integral \SPM.

	This concludes the proof: any optimal solution to the integral path model is 
	also optimal for the integral general model as there is a feasible solution 
	to the integral path model that solves the integral general model to 
	optimality and each solution to integral \SPM ~is feasible for integral \SGM.
	Now, we can use the result of \cite[Theorem 8, proof]{disser} to draw an 
	analogous conclusion for our general model: the optimal objective value 
	$\foptgm$ of integral \SGM ~is at most two if no two arc-disjoint paths 
	exist in $G'$ and is at least three if two arc-disjoint paths exist in $G'$.
	Deciding whether two such paths exist is NP-hard so that the claim follows. 
\end{proof}

\subsection{Proof of \texorpdfstring{\Cref{lemma:reducedflowproperty}}{Lemma \ref{lemma:reducedflowproperty}}}\label{prooflemma:reducedflowproperty}

\begin{proof}
  \newcommand{\tx}{\tilde{x}}
  Let $\tildex$ be a feasible solution to \SGM ~that does not satisfy 
  \eqref{eq:reducedflowproperty} for all nodes $v\in V\setminus\{s,t\}$. 
  Recall that we want to construct a feasible solution $x$ satisfying
  \begin{enumerate}
  	\item[(i)] ${\sum_{P:a\in \barP}x(P)\leq \sum_{P:a\in \barP}\tildex(P)}$
  	for all ${a\in A}$,
  	\item[(ii)] ${x(P)=\tildex(P)}$ for all ${P\in\dinPt}$, implying
  	  ${\fnom(x)=\fnom(\tildex)}$ and ${f_R(x)=f_R(\tildex)}$,
  \end{enumerate}
  and
  \begin{align}
  	\sum_{\substack{P\in\dinP(v):\\a\in P}}x(P)\leq \sum_{P\in\doutP(v)}x(P)
  	\quad\quad \forall a\in A
  	\tag{\ref{eq:reducedflowproperty}}
  \end{align}
  at every node $v\in V\setminus\{s,t\}$.
  
  As $G$ is a DAG, we can order the set of nodes 
  ${V=\{s=v_0,v_1,\dots,v_n,t=v_{n+1}\}}$ such that there is no directed 
  path in $G$ from $v_i$ to $v_j$ if $i>j$. 
  Additionally, there is a strict partial order 
  on the arcs such that $a=(v,w)\prec a'=(v',w')$ if $w=v'$ or if there is a 
  directed path from $w$ to $v'$ in $G$. 
  Let $\ell\in[n+1]$ be the minimal number such that there exists a feasible 
  solution that satisfies (i), (ii) and \eqref{eq:reducedflowproperty} at all 
  nodes $v_i$ with $i\in[n]$ and
  $i\geq \ell$, and we assume for contradiction that $\ell>1$. 

  We fix the node $v:=v_{\ell-1}$ for the remainder of this proof
  and introduce some further notation.
  Let $x$ be any feasible solution to \SGM.
  We denote the amount of incoming and outgoing flow at $v$ by 
  \begin{equation*}
  	f^-(x):=\sum_{P\in\dinP(v)}x(P) ~\text{ and }~ 
  	f^+(x):=\sum_{P\in\doutP(v)}x(P),
  \end{equation*}
  respectively.
  For an arc $a\in A$, we denote the set of paths that carry flow to $v$ 
  and use $a$ by 
  \begin{equation*}
  	\barP(x,a):=\{P\in\dinP(v)\defsep a\in P,~x(P)>0\}.
  \end{equation*}
  Additionally, for every arc we denote the flow ending at $v$ and using or not 
  using $a$ by
  \begin{align*}
  	y(x,a):=\sum_{P\in \barP(x,a)}x(P)
    	=\sum_{\substack{P\in\dinP(v):\\a\in P}}x(P) ~\text{ and }~ 
  	y^c(x,a):=\sum_{\substack{P\in\dinP(v):\\a\notin P}}x(P),
  \end{align*}
  respectively.
  The maximum of $y(x,a)$ over the arcs is denoted by $\bar{y}(x):=\max_{a\in A}y(x,a)$
  and we define the corresponding set of maximizing arcs by
  \begin{align*}
  	\bar{A}(x)&:=\{a\in A\defsep y(x,a)=\bar{y}(x)\}
  \end{align*}
  with $k(x):=|\bar{A}(x)|$. 
  With this notation, \eqref{eq:reducedflowproperty} reads
  \begin{align}\label{eq:reducedflowproperty-y}
  	\bar{y}(x)\leq f^+(x).
  \end{align}
  Robust flow conservation
  \eqref{static:general-rob-constraint} at $v$ reads 
  \begin{align}
    y^c(x,a)\geq f^+(x), \quad \forall a\in A,\label{app:robconservation}
  \end{align}
  since $\Gamma=1$. 
  Furthermore, we have 
  \begin{align}
    f^-(x)=y(x,a)+y^c(x,a),\quad \forall a\in A.\label{app:f-decomp}
  \end{align}

  Let $x'$ be one of the solutions that satisfy (i), (ii) and 
  \eqref{eq:reducedflowproperty} at the nodes $v_i$, $i\in[n]$ and $i\geq\ell$,
  and which, among them, minimizes the value
  \begin{align*}
    R(x'):=\sum_{a\in A}\max\{y(x',a)-f^+(x'), ~0\}>0.
  \end{align*}  
  
  We first show that it holds for $x'$  that
  \begin{itemize}
     \item[(a)] $f^-(x')=\bar{y}(x')+f^+(x')$ and
     \item[(b)] for all $a',a''\in\bar{A}(x')$, there exists a path $P\in\dinP(v)$ 
     with $x'(P)>0$ and $a',a''\in P$.
  \end{itemize}
  We arbitrarily choose an arc $a'\in\bar{A}(x')$.
  Due to the assumption that $x'$ minimizes $R(\cdot)$, we know that 
  modifying $x'$ only by reducing flow on paths in $\barP(x',a')$ results 
  in a flow that is infeasible in terms of robust flow conservation at $v$.
  The previous implies the following for every $P\in\barP(x',a')$: 
  there exists an arc $a''\notin P$ such that by reducing flow on $P$ by 
  any positive amount we would violate robust flow conservation 
  \eqref{app:robconservation} at $v$ for the scenario in which $a''$ fails.
  Hence, in the scenario in which $a''$ fails robust flow conservation is 
  fulfilled with equality, i.e. $y^c(x',a'')=f^+(x')$. 
  Now, we have
  $\bar{y}(x') = y(x', a')$ since $a'\in\bar{A}(x')$, $y^c(x',a')\geq f^+(x')$ 
  due to robust feasibility of $x'$, 
  the equality \eqref{app:f-decomp} for $a'$ and $a''$, 
  ${y(x', a'')\leq\bar{y}(x')}$ and $y^c(x',a'')=f^+(x')$.
  Combining this in the same order, we obtain
  \begin{align}\label{eq:equality_y_bar}
    \bar{y}(x')+f^+(x')\leq y(x',a')+y^c(x',a')=f^-(x')
      =y(x',a'')+y^c(x',a'')\leq \bar{y}(x')+f^+(x').
  \end{align}
  Thus, (a) holds. 
  
  Furthermore, it follows that $y(x',a'')=\bar{y}(x')$, i.e. $a''\in\bar{A}(x')$ 
  as well as $y^c(x',a')=f^+(x')$, i.e., robust flow conservation is satisfied 
  with equality in the scenario in which $a'$ fails. 
  Additionally, since \eqref{eq:reducedflowproperty-y} is not satisfied, the above 
  equations yield $y^c(x',a')=f^+(x')<y(x',a'')=\bar{y}(x')$. 
  Hence, there exists at least one path that ends at $v$, carries flow, and 
  contains $a'$ as well as $a''$ as otherwise $y^c(x',a')\geq y(x',a'')$ 
  would hold, which would contradict the former inequality. 
    
  Since $a'\in\bar{A}(x')$ has been chosen arbitrarily, 
  it follows that for any $a',a''\in\bar{A}(x')$ 
  there exists a path $P\in \dinP(v)$ with $a',a''\in P$ and $x'(P)>0$, 
  i.e. (b) holds.
  
  Note that for every feasible solution $x$ 
  with $x(P)=x'(P)$ for all $P\notin\dinP(v)$ and ${y(x,a)=y(x',a)}$ for all 
  $a\in A$
  the properties (a) and (b) hold and additionally $\bar{A}(x)=\bar{A}(x')$.\\
  Since $G$ is a DAG, property (b) implies that the strict partial order on $A$ induces 
  a strict total order on $\bar{A}(x')$ 
  and in particular that there exists a path $P'\in\dinP(v)$ such that 
  $\bar{A}(x')\subseteq P'$ and $y(x',a)>0$ for all $a\in P'$.

  There is a path decomposition $x_{pd}$ of $y(x',\cdot)$ with $x_{pd}(P'')>0$ 
  for some path $P''\in\dinP(v)$ with $P'\subseteq P''$ due to the following:
  the vector $y(x',\cdot)$ can be extended to a nominal $s$-$v$-flow with strict 
  flow conservation on an extended graph
  by introducing some artificial arcs $a_w=(s,w)$ for all $w\in V\setminus\{s,v,t\}$
  and defining 
  \begin{align*}
    y(x',a_w):= \sum_{a\in\doutA(w)}y(x',a)-\sum_{a\in\dinA(w)}y(x',a).
  \end{align*}
  Thus, on the extended graph there exists such a path decomposition which obviously
  induces a path decomposition of $y(x',\cdot)$ in $G$.  
  
  We now define a new feasible solution $x''$ to \SGM ~by
  \begin{align*}
    x''(P):=
      \begin{cases}
        x'(P),&P\notin\dinP(v),\\
        x_{pd}(P),&P\in\dinP(v).      
      \end{cases}
  \end{align*}
  This implies that $x''$ still satisfies (a) and that 
  $x''(P'')>0$ for the path $P''$ with $\bar{A}(x')=\bar{A}(x'')\subset P''$.
  Hence, we can reduce the flow on $P''$ by some positive amount $\epsilon>0$
  since for all $a\notin P''$ it holds that $y(x'',a)<\bar{y}(x'')$ 
  and thus $y^c(x'',a)>f^+(x'')$ due to (a) and \eqref{app:f-decomp}. 
  By this reduction we obtain a new feasible solution $x'''$ 
  which still satisfies (i), (ii) and 
  \eqref{eq:reducedflowproperty} at the nodes $v_i$, $i\in[n]$ and $i\geq\ell$.
  Additionally, due to $|\bar{A}(x')|=|\bar{A}(x'')| = k(x')$, we have
  \begin{align*}
    R(x''')\leq R(x')-k(x')\varepsilon < R(x'),
  \end{align*}   
  contradicting the assumption that $x'$ minimizes $R(\cdot)$.
  This contradicts the initial assumption $\ell>1$ so that
  there exists a solution $x$ satisfying (i), (ii), 
  and \eqref{eq:reducedflowproperty} at every node $v\in V\setminus\{s,t\}$.
\end{proof}

\subsection{Proof of \texorpdfstring{\Cref{Prop:gapNom_SPMNom}}{Proposition \ref{Prop:gapNom_SPMNom}}}\label{proofProp:gapNom_SPMNom}
\begin{proof}
	For an arbitrary $\Gamma\geq 2$, we construct an instance $I=(G,\Gamma)$ 
	that consists of four nodes $\{s,v_1,v_2,t\}$, 
	cf. \Cref{fig:gapNOM_SPMnom}.
	Let 
  	\begin{align*}
		\eta:=\frac{\Gamma(\alpha-2)+\alpha}{(\Gamma-1)(\alpha-2)}\in \Q.
	\end{align*}
  We note that $0<\eta\leq 1$.
	There are $\Gamma-1$ parallel arcs with capacity $1-\eta$ 
	from $s$ to $v_1$ $(a'_i)$, from $v_1$ to $v_2$ $(a''_i)$ and from 
	$v_2$ to $t$ $(a'''_i)$, where $i\in [\Gamma-1]$. 
	Besides, there are two additional arcs, $a'_{\Gamma}$ from $s$ to $v_1$ 
	and $a'''_{\Gamma}$ from $v_2$ to $t$ with unit capacity. 
	Furthermore, there is arc $a_1$ from $s$ to $v_2$ and arc $a_2$ from
	$v_1$ to $t$, each with capacity equal to $\Gamma-(\Gamma-1)\eta$.
  
	The nominal optimal  solution $\xoptnom$ fully utilizes the capacity of 
	all paths $P^2_i=\{a'_i,a_2\},~ i\in [\Gamma]$ and 
	$P^1_i=\{a_1,a'''_i\},~ i\in [\Gamma]$, resulting in an objective value 
	$\foptnom = 2(\Gamma-(\Gamma-1)\eta)$. 
	This is a nominal optimal flow solution, as 
	$\foptnom=\sum_{a\in\dinA(t)} u_a$. 
 
  	Next, we describe a robust optimal solution $\xoptpm$ to \SPM ~and prove 
  	optimality.
  	We set ${\xoptpm(P'_i)=1-\eta}$ for all paths $P'_i=\{a'_i,a''_i,a'''_i\},
  	~i\in [\Gamma-1]$. 
  	Besides, let $\xoptpm(P_1)=\xoptpm(P_2)=1$ for $P_1=\{a_1,a'''_{\Gamma}\}$ and 
  	$P_2=\{a'_{\Gamma},a_2\}$. 
  	In the worst case, the flow on the arcs $a_1,~a_2$ and 
  	$\Gamma-2$ of the arcs $a'''_i, ~i\in[\Gamma-1]$ is deleted, resulting in an optimal value 
  	of~$1-\eta$. 
    Since $\xoptpm$ saturates all arcs except $a_1$ and $a_2$, 
    an increase of the flow on one of these two arcs requires a decrease on one 
    of the paths $P'_i$ to satisfy the capacity constraints. 
  	However, the paths containing $a_1$ and $a_2$ carry the largest amounts of flow
  	and therefore are deleted in the original worst case scenarios 
  	and also the resulting worst case scenarios. 
  	Thus, such an increase on $a_1$ or $a_2$ would decrease the objective value 
  	and hence the above flow solution is robust optimal
  	with maximal nominal value. 
	The nominal value of this solution is $\foptpmnom=2+(\Gamma-1)(1-\eta)$.
  
 	Combining the above proves the first assertion since
  \begin{align*}
    	\frac{\foptnom}{\foptpmnom}
    	= \frac{2(\Gamma-(\Gamma-1)\eta)}{2+(\Gamma-1)(1-\eta)} 
    	=\frac{2(\Gamma-\Gamma-\frac{\alpha}{(\alpha-2)})}
    	  {1+\Gamma-\Gamma-\frac{\alpha}{(\alpha-2)}}
      = \frac{\frac{-2\alpha}{(\alpha-2)}}
    		{1-\frac{\alpha}{(\alpha-2)}}
    	= \alpha.
	\end{align*}
	We observe that $|\dinA(v)|\leq \Gamma$ holds for all $v\in V$. 
	Hence, it directly follows from \Cref{cor:static:aux-lemma} that 
	the optimal solution $\xoptpm$ to \SPM ~with maximal nominal value is
	also an optimal solution to \SGM ~maximizing the nominal value, yielding
	the second assertion.
  
	We further note that the constructed instance contains arcs with non-integer 
	capacities. 
	However, since $\alpha$ and therefore $\eta$ are rational, the capacities
	are all rational and we can w.l.o.g. scale the instance to an instance with 
	integer capacities.
	\begin{figure}[h]
    \begin{tikzpicture}[scale=.9]
      \node(s) at (-3.5,0) [circle,draw=black, thick, minimum size=8mm]{$s$};
      \node(v1) at (0,1.7) [circle,draw=black, thick, minimum size=8mm]{$v_1$};
      \node(v2) at (3.5,0)[circle,draw=black,thick,minimum size=8mm]{$v_2$};
      \node(t) at (7,1.7) [circle,draw=black, thick, minimum size=8mm] {$t$};
      \node at (-1.68,0.75)[fill, circle, inner sep = 0.5pt] {};
      \node at (-1.75,0.85)[fill,circle,inner sep=0.5pt]{};
      \node at (-1.82,0.95)[fill, circle, inner sep = 0.5pt] {};
      \node at (1.68,0.75)[fill, circle, inner sep = 0.5pt] {};
      \node at (1.75,0.85)[fill,circle,inner sep=0.5pt]{};
      \node at (1.82,0.95)[fill, circle, inner sep = 0.5pt] {};
      \node at (5.18,0.95)[fill, circle, inner sep = 0.5pt] {};
      \node at (5.25,0.85)[fill,circle,inner sep=0.5pt]{};
      \node at (5.32,0.75)[fill, circle, inner sep = 0.5pt] {};
      \path[thick,->,bend left=40]([xshift=0.05cm]s.north)
      			edge node[above,rotate=33]{$\left(a'_{\Gamma},1\right)\hs{2em}$}
      			([xshift=-0.25cm]v1.north);
      \path[thick,->]([yshift=0.25cm]s.east)edge node[above,rotate=33]
      			{$\left(a'_{\Gamma-1},1-\eta\right)\hs{0.85em}$}
      			([yshift=0.25cm]v1.west);
      \path[thick,->]([yshift=-0.25cm]s.east)edge node[below,rotate=33]
      			{$\hs{1.6em}\left(a'_1,1-\eta\right)$}
      			([yshift=-0.25cm]v1.west);
      \path[thick,->]([xshift=0.25cm]s.south)edge node[below]
      			{$\left(a_1,\Gamma-(\Gamma-1)\eta\right)$}
      			([xshift=-0.25cm]v2.south);
      \path[thick,->]([xshift=0.25cm]v1.north)edge node[above]
      			{$\left(a_2,\Gamma-(\Gamma-1)\eta\right)$}
      			([xshift=-0.25cm]t.north);
      \path[thick,->]([yshift=0.25cm]v1.east)edge node[above,rotate=-33]
      			{$\hs{1em}\left(a''_1,1-\eta\right)$}
      			([yshift=0.25cm]v2.west);
      \path[thick,->]([yshift=-0.25cm]v1.east)edge node[below,rotate=-33]
      			{$\left(a''_{\Gamma-1},1-\eta\right)\hs{1.3em}$}
      			([yshift=-0.25cm]v2.west);
      \path[thick,->]([yshift=0.25cm]v2.east)edge node[above,rotate=33]
      			{$\hs{0.4em}\left(a'''_1,1-\eta\right)$}
      			([yshift=0.25cm]t.west);
      \path[thick,->]([yshift=-0.25cm]v2.east)edge node[below,rotate=33]
      			{$\hs{0.85em}\left(a'''_{\Gamma-1},1-\eta\right)$}
      			([yshift=-0.25cm]t.west);  
      \path[thick,->,bend right=40]([xshift=0.25cm]v2.south)
      		edge node[below,rotate=33]{$\hs{0.3em}\left(a'''_{\Gamma},1\right)$}
      			([xshift=-0.05cm]t.south);	
    \end{tikzpicture}
    \caption{Instance with $\foptnom\hs{-0.2em}\geq\hs{-0.2em}\foptpmnom$ 
    	and $\foptnom\hs{-0.2em}\geq\hs{-0.2em}\foptgmnom$  
    	in the proof of Prop.~\ref{Prop:gapNom_SPMNom}.}
    \label{fig:gapNOM_SPMnom}
	\end{figure}
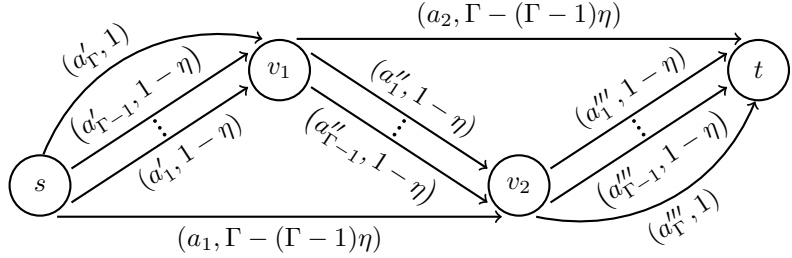
\end{proof}

\section{Proofs of \texorpdfstring{\Cref{sec:dynamic}}{Section \ref{sec:dynamic}}}
\subsection{Proof of \texorpdfstring{\Cref{lem:DAM/DGM:pw-constant}}{Proposition \ref{lem:DAM/DGM:pw-constant}}}\label{prooflem:DAM/DGM:pw-constant}
\begin{proof}
	The proof works analogously to the proof of 
	\cite[Proposition~1]{gottschalk}.
	
	(i)	Let $\tildex$ be an optimal solution to the continuous 
	dynamic general model. 
	In order to 
	generate a piecewise constant function from the solution $\tildex$ for each 
	path $P\in\barP$, the time interval $[0,T)$ is cut into unit intervals 
	$[b-1,b)~\forall b \in \T$. 
	Define $x(P,\theta) \fa P\in\barP$ as follows:
	\begin{displaymath}
	x(P,\theta):=\left\{\begin{array}{ll}\int\limits_{b-1}^{b}\tildex(P,t)~dt, 
	&\theta \in [b-1,b), b\in\T\\ 0,&\text{else}. 
	\end{array}\right. 
	\end{displaymath}
	By definition, the objective values of $\tildex$ and $x$ are equal and the 
	non-negativity constraint is satisfied.
	We show by contradiction that $x$ fulfills the capacity constraint. 
	If the amount of flow on an arbitrary arc $a$, $\sum_{{P\in\barP: a\in P}} 
	x(P,\theta_{z,a}(P))$, violates the capacity constraint 
	for any point in time $\theta \in [b-1,b)$, then 
	it violates the capacity constraint  for all 
	$\theta \in [b-1,b)$ due to the integrality of $\tau_a$ and $\Delta \tau_a$.
	As for each path $P\in\barP$, the flow $x(P,\theta)$ is defined as 
	integral of $\tildex$ in the time 
	interval $[b-1,b)$ of length $1$, there is at least one point in time,
	$\theta \in [b-1,b)$, such that $\tildex(P,\theta)$ is
	greater than or equal to $x$. 
	Hence, the capacity constraint would be violated by $\tildex$ as well. 
	As a consequence, since the capacity constraint is fulfilled by 
	$\tildex$ for all points in time, it is also satisfied by $x$.
	Analogous arguments can be employed in order to show that the piecewise 
	constant solution $x$ fulfills robust flow conservation.
	
	(ii) The proof can be conducted analogously to the proof of 
	part (i). 
\end{proof}

\subsection{Proof of \texorpdfstring{\Cref{thm:DAM:reform}}{Theorem \ref{thm:DAM:reform}}}\label{proofthm:DAM:reform}
\begin{proof}
	Every robust feasible solution $x$ to \DAM ~satisfies 
	\eqref{dynamic:arc:rob-constraint}, which can be equivalently reformulated to
	\begin{equation}\label{thm:DAM:xfeasible1}
		\min_{z\in\Lambda}
			\left\{\sum_{a\in\dinA(v)}x(a,\theta-\tau_a-z(a)\Delta\tau_a)\right\}
    		\geq\sum_{a\in\doutA(v)} x(a,\theta), \quad
        	\forall v\in V \setminus\{s,t\}, \theta\in \T.
	\end{equation}
	For any $a\in A$, $z\in\Lambda$ and $\theta\in\T$, it holds that 
	\begin{align*}
	  x(a,\theta-\tau_a-z(a)\Delta\tau_a)
	    =x(a,\theta-\tau_a)-
	    z(a)(x(a,\theta-\tau_a)-x(a,\theta-\tau_a-\Delta\tau_a)).
	\end{align*}
	Thus, constraint \eqref{thm:DAM:xfeasible1} is equivalent to
	\begin{equation}\label{thm:DAM:xfeasible}
		\sum_{a \in \dinA(v)}x(a,\theta-\tau_a)-y_{v,\theta}(x)\geq 
			\sum_{a\in\doutA(v)} x(a,\theta), \qquad \qquad \qquad
			\forall v\in V\setminus \{s,t\}, \theta \in \T,
	\end{equation}
	with 
  	\vspace{-0.3cm}
	\begin{align*}
       y_{v,\theta}(x):=
        	&\max_{z_{v,\theta}\in\{0,1\}^{|\dinA(v)|}} \sum_{a\in\dinA(v)}
        		z_{v,\theta}(a)~(x(a,\theta-\tau_a)-x(a,\theta-\tau_a-
        		\Delta\tau_a))\\[3pt]
        	&\qquad\quad \st\qquad \sum_{a\in\dinA(v)}z_{v,\theta}(a)\leq\Gamma.
  	\end{align*}
  Since $x$ is a fixed parameter in this maximization problem, 
  ${z_{v,\theta}\in\{0,1\}^{|\dinA(v)|}}$ can be relaxed to 
  ${z_{v,\theta}\in[0,1]^{|\dinA(v)|}}$
  without changing the objective value $y_{v,\theta}(x)$. 
  Hence, this relaxed problem can be dualized and we obtain
  \begin{subequations}\label{thm:DAM:yvtheta}
    \begin{align}
    	y_{v,\theta}(x)=~ \min_{\eta,\lambda} \quad
    	  &\sum_{a\in\dinA(v)}\lambda_{\theta}(a)+\Gamma\eta_{v,\theta}\\[5pt]
        \st \quad &\eta_{v,\theta}+\lambda_{\theta}(a) \geq x(a,\theta-\tau_a) 
        -x(a,\theta-\tau_a-\Delta\tau_a), \quad &\forall a\in\dinA(v),\\
      &~\lambda_{\theta}(a) \geq 0, \quad &\forall a\in\dinA(v),\\
      &~\eta_{v,\theta} \geq 0.
    \end{align}
  \end{subequations}

	Similarly, for a feasible and fixed $x$ we can reformulate 
	the inner optimization problem of the objective function of \DAM, i.e.
	\begin{align}
	\min_{z\in\Lambda} \sum_{a\in\dinA(t)}
		\sum_{\theta\in [T-\tau_a-z(a)\Delta\tau_a]}x(a,\theta)
     ~=~ \sum_{a\in\dinA(t)}\sum_{\theta\in\T}~x(a,\theta-\tau_a)-y_t(x)	
			  \label{thm:DAM:xoptimal}	 
  \end{align}
  with 
  \vspace{-0.4cm}
  \begin{align*}
        y_t(x):= 
        	&\max_{z\in\{0,1\}^{|\dinA(t)|}} \sum_{a\in\dinA(t)}z(a) 
        		\left(\sum_{i\in [\Delta\tau_a]}x(a,T-\tau_a-(i-1))\right)\\[3pt]
        	& \qquad \st \qquad \sum_{a\in\dinA(t)}z(a)\leq\Gamma.
  \end{align*}
  Again, for the fixed $x$, the integrality of $z$ can be relaxed to 
  ${z\in[0,1]^{|\dinA(t)|}}$. Dualizing the relaxed problem gives
  \begin{subequations}\label{thm:DAM:yt}
  \begin{align}
    y_t(x)=~ &\min_{\mu,\nu} 
      \sum_{a\in\dinA(t)}\nu(a) + \Gamma \mu \\
    & \st \quad \mu+\nu(a)\geq 
      \sum_{i\in [\Delta\tau_a]}x(a,T-\tau_a-(i-1)),\quad 
      &\forall a\in\dinA(t),\\
    & \qquad \quad~ \nu(a) \geq 0, \qquad &\forall a\in\dinA(t), \\
    & \qquad \quad~ \mu \geq 0.
    \end{align}
  \end{subequations}
    
  A substitution of \eqref{thm:DAM:yvtheta} into \eqref{thm:DAM:xfeasible} 
  and \eqref{thm:DAM:yt} into the right hand side of \eqref{thm:DAM:xoptimal} 
  gives problem~\eqref{dynamic:arc-reform}. 
  This is a linear program with a polynomial number of constraints and a 
  polynomial number of variables, which concludes the proof.
\end{proof}